\documentclass{amsart}
\usepackage{amsmath}
\usepackage{amsfonts}
\usepackage{amssymb}
\usepackage{epsfig}
\newtheorem{theorem}{Theorem}[section]

\newtheorem{lm}[theorem]{Lemma}

\newtheorem{cor}[theorem]{Corollary}

\newtheorem{rem}[theorem]{Remark}
\newtheorem{pr}[theorem]{Proposition}
\newtheorem{example}[theorem]{Example}

\usepackage{color}

\newtheorem{question}[theorem]{Question}
\newtheorem{problem}[theorem]{Problem}
\newtheorem{ex}[theorem]{Example}

\begin{document}

\title{On the notion of Krull super-dimension}
\author{A.Masuoka}
\address{Institute of Mathematics, University of Tsukuba, Ibaraki 305-8571, Japan}
\email{akira@math.tsukuba.ac.jp}
\author{A. N. Zubkov}
\address{Department of Mathematical Science, UAEU, Al-Ain, United Arabic Emirates; Sobolev Institute of Mathematics, Omsk Branch, Pevtzova 13, 644043 Omsk, Russian Federation}
\email{a.zubkov@yahoo.com}
\begin{abstract}
We introduce the notion of Krull super-dimension of a super-commutative super-ring. This notion is used to describe regular super-rings and calculate Krull super-dimensions of completions of super-rings. Moreover, we use this notion to introduce the notion of  super-dimension of any irreducible superscheme of finite type. Finally, we describe nonsingular superschemes in terms of sheaves of K\"{a}hler superdifferentials.  
\end{abstract}
\maketitle

\section*{introduction}

The concept of Krull dimension is one of the most fundamental concepts of the theory of commutative rings. On the basis of this concept, one can define the dimension of an algebraic variety or, more generally, the dimension of a scheme. In contrast, so far no one has defined the concept of super-dimension of a superscheme, which would be naturally derived from the properties of its sheaf of "super-functions".     
 
Motivated by the above, we introduce the concept of Krull super-dimension of a (super-commutative) Noetherian super-ring. We develop some fragment of dimension theory of Noetherian super-rings. For example, we show that any finitely generated superalgebra $A$ contains a polynomial super-subalgebra $C$ of the same Krull super-dimension, such that $A$ is a finitely generated $C$-supermodule. This result can be regarded as a super Noether Normalization Theorem. 

Further, we investigate local Noetherian super-rings and give a criteria when such a super-ring is regular (Proposition \ref{prop:regular_local} and Theorem \ref{thm:regular_local}). We generalize this criteria for all (not necessary local) Noetherian super-rings (Proposition \ref{prop:regular_global}). Note that our definition of regularity is based on the notion of Krull super-dimension, and we do not use the notion of regular sequence as in \cite{sm}, but our results coincide with the results therein (see Remark \ref{coincidenceofdefinitions}). Finally, we show how the super-dimension of a completion of Noetherian super-ring is determined by the super-dimensions of its localizations (Theorem \ref{super-dim of completion}).   

We also describe certain local Noetherian superalgebras in terms of K\"{a}hler superdifferentials
(Theorem \ref{regularity}). This result is used to describe nonsingular irreducible superschemes of finite type over a perfect field. More precisely, such a superscheme $X$ is nonsingular if and only if its sheaf of K\"{a}hler superdifferentials $\Omega_{X/K}$ is a locally free sheaf of $\mathcal{O}_X$-supermodules of rank equal to the super-dimension of $X$ (Theorem \ref{acriteriaofnonsingularity}).

In connection with the above result, one can note a new phenomena in the theory of superschemes, that does not take place in the theory of schemes. First, there are integral superschemes, which are singular at any point. Second, an integral superscheme $X$ may contain a proper closed integral super-subscheme $Y$ of the same super-dimension! Nevertheless, we show that if both $X$ and $Y$ are integral and {\it generically nonsingular}, then the coincidence of their super-dimensions implies $X=Y$ (Theorem \ref{coincidence}).  

The paper is organized as follows. In the first section we collect the necessary results on super-rings and supermodules. The second section is devoted to studying superschemes and sheaves of supermodules over them. In the third section we recall the definition of supermodule of relative differential forms and formulate some standard properties of it. These properties are generalized for sheaves of K\"{a}hler superdifferentials. 

The Krull super-dimension is introduced in the fourth section. As it has been mentioned above, we prove a super Noether Normalization Theorem and give a simple algorithm how to calculate the odd Krull dimension of a given finitely generated superalgebra (Lemma \ref{odd-Krull}). Using this algorithm, we construct a superalgebra $A$ and its quotient $A/I$, such that {\bf the odd Krull dimension of $A/I$ is greater than the odd Krull dimension of $A$}! We also discuss the case of one relation superalgebra and calculate the super-dimension of a completion of a Noetherian super-ring (Theorem \ref{super-dim of completion}.)

The fifth section is devoted to regular Noetherian super-ring (see above). In the sixth section we introduce the notion of super-dimension of an irreducible superscheme of finite type over a field. We characterize nonsingular superschemes, as well as their closed nonsingular super-subschemes, in terms of their sheaves of K\"{a}hler superdifferentials.

In the seventh section we formulate some questions and open problems, those would stimulate the further progress in the dimension theory of super-rings and superschemes. 

\section{Super-rings and supermodules}

\subsection{Super-rings}

A $\mathbb{Z}_2$-graded ring $R$ (with unity) is called a \emph{super-ring}. Let $r\mapsto |r|$ be a \emph{parity function} on the set of non-zero homogeneous elements of $R$, i.e. $|r|=i$ if and only if $r\in R_i, i\in\mathbb{Z}_2$. A homogeneous non-zero element $r$ is called \emph{even}, provided $|r|=0$, otherwise $r$ is called \emph{odd}.  

In what follows, all homomorphisms of super-rings are supposed to be graded, unless otherwise stated.

\subsection{Supermodules}

Let $R$ be a super-ring. A left $\mathbb{Z}_2$-graded $R$-module $M$ is called an $R$-\emph{supermodule}. The category of $R$-supermodules with graded morphisms (as well as the category of right $R$-supermodules), denoted by $_{R}\mathrm{mod}$ (respectively, by $\mathrm{mod}_R$), is obviously abelian.  

If $M$ is an $R$-supermodule, then let $m\mapsto |m|$ be a parity function on the set of non-zero homogeneous elements of $M$, i.e. $|m|=i$ if and only if $m\in M_i, i\in\mathbb{Z}_2$. As above, a non-zero homogeneous element $m$ is called \emph{even}, provided $|m|=0$, otherwise $m$ is called \emph{odd}.

\subsection{Super-commutative super-rings}

A super-ring $R = R_0\oplus R_1$ is said to be \emph{super-commutative}, if the following are satisfied:
\begin{itemize}
\item[(i)] $rs = sr$,\ either if\ $r, s \in R_0$,\ or if\ $r\in R_0$ and $s \in R_1$; 
\item[(ii)] $s^2= 0$,\ if\ $s \in R_1$. 
\end{itemize}
These two conditions are equivalent to $rs=(-1)^{|r||s|}sr$, provided $2$ is not zero divisor in $R$. On the other hand, if $2$ is a zero divisor in $R$, then $R$ is super-commutative if and only if $R$ is a commutative ring, which satisfies (ii).
\begin{example}\label{ex:exterior}
\emph{Let $B$ be a commutative ring, and let $M$ be an $B$-module. We let $\wedge_B(M)$ denote the exterior $B$-algebra on $M$. Supposing that all elements in $M$ are odd,
we regard $\wedge_B(M)$ as a super-commutative $B$-superalgebra. This is indeed the quotient 
of the tensor $R$-algebra $T_B(M)$ on $M$, which is an $B$-superalgebra, 
divided by the relation $m^2=0$, $m \in M$. If $M$ is finitely generated projective, then the
canonical map $M \to \wedge_B(M)$ is an injection, and $\wedge_B(M)$ is finitely generated projective
as an $B$-module.
If $M$ has finite $B$-free basis $Y_1,\dots, Y_s$, we write
\[ B[Y_1,\dots, Y_s] \]
for $\wedge_B(M)$, and call it the} polynomial $B$-superalgebra \emph{in odd indeterminants  $Y_1,\dots Y_s$.} 
\end{example}
\begin{example}\label{polynomialsuperalgebra}
\emph{More generally, for any superalgebra $A$ one can define a} polynomial $A$-superalgebra \emph{$A[X_1, \ldots X_m\mid Y_1, \ldots Y_n]$ in $m$ even indeterminants $X_1, \ldots , X_m$ and $n$ odd indeterminants $Y_1, \ldots, Y_n$, as 
\[A\otimes K[X_1, \ldots, X_m][Y_1, \ldots, Y_n]\]. If it does not lead to confusion, we use the shorter notation $A[X\mid Y]$.}
\end{example}

If $R$ is super-commutative, then any left $R$-supermodule $M$ can be regarded as an right $R$-supermodule, by setting $mr=(-1)^{|r||m|}rm, r\in R, m\in M$, and vice versa. In other words, the categories $_{R}\mathrm{mod}$ and $\mathrm{mod}_R$ are naturally isomorphic. Moreover, the category $_{R}\mathrm{mod}$ is a tensor category with a braiding
\[t_{M, N} : M\otimes_R N\simeq N\otimes_R M, m\otimes n\mapsto (-1)^{|m||n|}n\otimes m, m\in M, n\in N.\]

From now on all super-rings are assumed to be super-commutative. 

Let $R$ be a super-ring. It is obvious that any left super-ideal $I$ of $R$ is also right, hence two-sided. A super-ideal $\mathfrak{p}$ is called \emph{prime} (respectively, \emph{maximal}), provided $R/\mathfrak{p}$ is an integral domain (respectively, a field). A \emph{localization} of an $R$-supermodule $M$ at a prime super-ideal $\mathfrak{p}$ is defined as $M_{\mathfrak{p}}=(R_0\setminus\mathfrak{p}_0)^{-1}M$.

Let $I_R$ denote the super-ideal $RR_1$, and let $\overline{R}$ denote the quotient-ring $R/I_R$. A super-ideal $\mathfrak{p}$ is prime (respectively, maximal) if and only if $I_R\subseteq\mathfrak{p}$ and $\mathfrak{p}/I_R$ is prime (respectively, maximal) ideal of $\overline{R}$.  

The intersection of all prime super-ideals of $R$ coincides with the largest nil super-ideal of $R$, called a \emph{nil-radical} of $R$ and it is denoted by $\mathrm{nil}(R)$. It is obvious that
$I_R\subseteq\mathrm{nil}(R)$.

A super-ring $R$ with a unique maximal super-ideal is said to be \emph{local}. For example, if 
$\mathfrak{p}$ is a prime super-ideal of $R$, then its \emph{localizaion} $R_{\mathfrak{p}}$ is a local supe-ring with the maximal super-ideal $R_{\mathfrak{p}}\mathfrak{p}$.

A morphism $\alpha : R\to S$ between local super-rings is said to be \emph{local} if $\alpha(\mathfrak{m})\subseteq\mathfrak{n}$, where $\mathfrak{m}$ and $\mathfrak{n}$ are the unique maximal super-ideals of $R$ and $S$, respectively. 

\subsection{Super-vector spaces}

If $R$ is a field $K$, then an object $V$ from $_{K}\mathrm{mod}$ 
is called a \emph{super-vector space} over $K$, and the \emph{super-dimension} $\mathrm{sdim}_K(V)$
of $V$ is defined by $\mathrm{sdim}_K(V) = r \mid s$, where $r= \dim_K(V_0)$, $s = \dim_K(V_1)$.

\subsection{Superdomains and superfields}

A super-ring $A$ is called \emph{reduced}, if the ring $\overline{A}=A/I_A$ is reduced, i.e.
$I_A=\mathrm{nil}(A)$. 
The following lemma is a folklore.
\begin{lm}\label{reduced}
A super-ring $A$ is reduced if and only if for any prime super-ideal $\mathfrak{p}$ of $A$, the super-ring $A_{\mathfrak{p}}$ is.
\end{lm}

We say that a super-ring $A$ is an \emph{integral superdomain} (or just {\it superdomain}), if $I_A$ is a prime super-ideal of $A$. If, additionally, the natural superalgebra morphism $A\to A_{I_A}$ is injective, then $A$ is a {\it strong} superdomain. The additional condition is
equivalent to that for any $s\in A_0\setminus A^2_1$ and any $a\in A$ the equality $sa = 0$ implies
$a = 0$. This, in turn, is equivalent to the apparently stronger condition that for any
$s\in A\setminus I_A$ and any $a\in A$ the equality $sa = 0$ implies $a = 0$. Indeed, if $sa = 0$, then
$s_0 a = −s_1 a$, whence $s_0(s_1 a) = s^2_1 a =0$. The former condition ensures first $s_1 a = 0$,
then $s_0 a = 0$, and finally $a = 0$.
 
A superring $A$ is said to be a {\it superfield} if any element $a\in A\setminus I_A$ is invertible, or equivalently, if any $a\in A_0\setminus A_1^2$ is invertible. In other words, $A$ is a superfield if and only if $A$ is a local super-ring with the unique maximal super-ideal $I_A$ if and only if $I_A$ is a maximal super-ideal. Obviously, a superfield is a strong superdomain.
For example, the polynomial superalgebra $\wedge_K(Y_1, \ldots, Y_r)$ over a field $K$ in odd indeterminants  $Y_1, \ldots , Y_r$, is a superfield.

 If $A$ is a superdomain, then $A_{I_A}$ is obviously a superfield that is called a {\it superfield of fractions} of $A$ and it is denoted by $SQ(A)$.

\subsection{Noetherian super-rings}

Recall that a super-ring $R$ is called \emph{Noetherian} if the super-ideals of $R$ satisfy ascending chain condition (ACC). As it has been proven in \cite{maszub1}, $R$ is Noetherian if and only if it is left or right Noetherian as a ring. 
\begin{lm}\label{Noetherian}
A super-ring $R$ is Noetherian if and only if $R_0$ is a Noetherian ring and the $R_0$-module $R_1$ is finitely generated.
\end{lm}
\begin{proof}
The "if" part. The assumptions imply that the $R_0$-submodules of $R$ satisfy the ACC, whence
$R$ is Noetherian. 

The "only if" part. Given an ascending chain
\[
\mathfrak{r}_0\subset \mathfrak{r}_1 \subset \dots 
\]
of ideals of $R_0$ (respectively, of $R_0$-submodules of $R_1$), we have the ascending chain 
\[
\mathfrak{r}_0\oplus \mathfrak{r}_0 R_1 \subset \mathfrak{r}_1\oplus \mathfrak{r}_1 R_1 \subset 
\dots \quad \text{(respectively,}\ 
\mathfrak{r}_0 R_1\oplus \mathfrak{r}_0 \subset \mathfrak{r}_1 R_1\oplus \mathfrak{r}_1 \subset 
\dots)
\]
of super-ideals of $R$. Therefore, if $R$ is Noetherian, then $R_0$ is Noetherian, and 
$R_1$ is a Noetherian $R_0$-module as well, hence finitely generated.
\end{proof}

\subsection{Completion}\label{subsec:completion}

Let $R$ be a super-ring and $M$ be an $R$-supermodule. If $I$ is a super-ideal of $R$, then $M$ can be endowed with the $I$-\emph{adic} topology so that a subset $U\subseteq M$ is open if and only if
$U=\cup_{j\in J}(m_j+I^{k_j}M)$ for some $m_j\in M, k_j\in\mathbb{Z}_{\geq 0}$. In other words, $R$ turns into a topological super-ring and $M$ turns into a topological $R$-supermodule with respect to their $I$-adic topologies.

Similarly, $M$, and $R$ as well, can be endowed with $I_0$-adic topologies, being regarded as $R_0$-modules. It is clear that the $I_0$-adic topologies of both $M$ and $R$ coincide with their $RI_0$-adic topologies.  
\begin{lm}\label{coincidenceoftopologies}
These two topologies coincide.
\end{lm}
\begin{proof}
It suffices to prove that for any integer $k> 0$, $I^k M\subseteq I_0^{[\frac{k}{2}]}M\subseteq I^{[\frac{k}{2}]}M$. The non-trivial first inclusion follows, by using $I_1^2\subseteq I_0$, as follows.  
\[I^k=\sum_{0\leq i\leq k}I_0^{k-i} I_1^i\subseteq\sum_{i \ even} I_0^{k-\frac{i}{2}}+\sum_{i \  odd}I_0^{k-\frac{i+1}{2}}I_1\subseteq I_0^{[\frac{k}{2}]}I.\] 
\end{proof}
The $I$-adic completion 
\[\widehat{R} = \underset{\longleftarrow}{\lim}\, R/I^n\]
of $R$ is naturally a topological super-ring, in which the super-ideals $\widehat{R}(k)=\ker(\widehat{R}\to R/I^k)$ form a base of neighborhhods of zero. 

We have the canonical map $R \to \widehat{R}$ of topological super-rings.
If it is an isomorphism, we say that $R$ is \emph{complete} with respect to its $I$-adic topology.

Finally, if $R$ is a local super-ring with a maximal super-ideal $\mathfrak{m}$, then we just say that $R$ is \emph{complete}, provided $R$ is complete with respect to its $\mathfrak{m}$-adic topology.

Similarly, the $I$-adic completion 
\[\widehat{M} = \underset{\longleftarrow}{\lim}\, M/I^n M\]
of an $R$-supermodule $M$ has a natural structure of a topological $\widehat{R}$-supermodule, in which the super-submodules $\widehat{M}(k)=\ker(\widehat{M}\to M/I^k M)$ form a base of neighborhhods of zero.  We also have the canonical map $M \to \widehat{M}$ of topological $R$-supermodules.
\begin{rem}\label{I_0-adic and I-adic}
\emph{Lemma \ref{coincidenceoftopologies} implies that 
\[\widehat{R}\simeq \underset{\longleftarrow}{\lim}\, R/I_0^n R,\]
and, therefore, the homogeneous components $\widehat{R}_i, i=0, 1$, are isomorphic to the $I_0$-adic completion of $R_i$. The similar statement holds for $R$-supermodules.}
\end{rem}
Observe that $\widehat{R}$ and $\widehat{M}$ are naturally isomorphic to $\widehat{R'}$ and $\widehat{M'}$, where $R'=R/\cap_{n\geq 0}I^n$ and $M'=M/\cap_{n\geq 0}I^n M$. Moreover, $R'$ and $M'$  are Hausdorff spaces with respect to their $I'$-adic topologies, where $I'=I/\cap_{n\geq 0}I^n$.  
Using this remark, one can easily superize Theorem 5 and Corollary 1, \cite{zs}, chapter VIII, as follows. 
\begin{lm}\label{apropertyofcompletion}
Let $M$ be a finitely generated $R$-supermodule. Then $\widehat{M}=\widehat{R}M$. 
\end{lm}
\begin{cor}
Assume additionally that $I$ is finitely generated. Then $\widehat{M}(k)=\widehat{I^k M}=I^k\widehat{M}$ and $\widehat{R}(k)=\widehat{I^k}=I^k \widehat{R}=\widehat{I}^k$ for any $k\geq 0$. 
\end{cor}
\begin{pr}\label{completionofNoetherian}
Let $R$ be a Noetherian super-ring and $I$ be a super-ideal of $R$. Then the functor $M\to\widehat{M}$, that takes a finitely generated $R$-supermodule to its $I$-adic completion, is exact. Moreover, $\widehat{R}$ is a Noetherian super-ring, hence $\widehat{M}$ is finitely generated whenever $M$ is.
\end{pr}
\begin{proof}
By Lemma \ref{Noetherian} any finitely generated $R$-supermodule is finitely generated as an $R_0$-module. It remains to  combine Lemma \ref{coincidenceoftopologies}, Remark \ref{I_0-adic and I-adic} and Lemma \ref{apropertyofcompletion} with Theorem 54 and 23(K), \cite{mats}.
\end{proof}
\begin{pr}\label{Hausdorff}
Let $R$ be a Noetherian super-ring, $I$ be a super-ideal of $R$, $M$ be a finitely generated $R$-supermodule. Then $\cap_{k\geq 0}I^k M$ consists of all elements $m\in M$ such that there is $x\in I_0$ with $(1-x)m=0$. In particular, if $R$ is a local super-ring with a unique maximal super-ideal $\mathfrak{m}$, then $\cap_{k\geq 0}\mathfrak{m}^k M=0$, that is $M$ is a Hausdorff space with respect to its $\mathfrak{m}$-adic topology. 
\end{pr}
\begin{proof}
As was already observed, $M$ is a finitely generated $R_0$-module and $\cap_{k\geq 0}I^k M=\cap_{k\geq 0}I_0^k M$. Then Krull intersection theorem (see \cite{bur}, III, \S 3, Proposition 5) concludes the proof.
\end{proof}
\begin{lm}\label{propertiesoflocal}
Let $R$ be a local Noetherian super-ring with a maximal super-ideal $\mathfrak{m}$. For its $\mathfrak{m}$-adic completion $\widehat{R}$ the following hold :
\begin{itemize}
\item[(1)] $\widehat{R}$ is a complete Noetherian local super-ring with maximal $\widehat{\mathfrak{m}}$. 
\item[(2)] $I_{\widehat{R}}$ coincides with $\widehat{I_R}=\widehat{R}R_1$. 
\item[(3)] $\overline{\widehat{R}} = \widehat{R}/I_{\widehat{R}}$ coincides with the $\overline{\mathfrak{m}}$-adic completion $\widehat{\overline{R}}$ of $\overline{R}$,
where $\overline{\mathfrak{m}}=\mathfrak{m}_0/R_1^2$. 
\end{itemize}
\end{lm}
\begin{proof}
The commutative ring $\widehat{R}_0$ is a $\mathfrak{m}_0$-adic completion of the local ring $R_0$, hence local. Moreover, the maximal ideal of $\widehat{R}_0$ coincides with the $\mathfrak{m}_0$-adic completion of $\mathfrak{m}_0$, that in turn coincides with $\widehat{\mathfrak{m}}_0$. Proposition \ref{completionofNoetherian} infers (1) and (3) as well.

Consider the image of the $\mathfrak{m}_0$-adic completion of the canonical map $R\otimes_{R_0} R_1\to R$, using Remark \ref{I_0-adic and I-adic}. It is $\widehat{I_R}$ on one hand, and is $I_{\widehat{R}}$ on the other hand. This proves (2).
\end{proof}

\section{Superschemes}

For the details of the content of this section we refer to \cite{ccf, man, maszub1, var, zub}.

\subsection{Geometric superspaces}

Recall that a \emph{geometric superspace} (or \emph{local ringed superspace}) $X$ consists of a topological space $X^e$ and a sheaf of super-commutative super-rings $\mathcal{O}_X$ such that all stalks $\mathcal{O}_{X, x}, x\in X^e$, are local super-rings. A morphism of
superspaces $f : X\to  Y$ is a pair $( f^e, f^*
)$, where $f^e : X^e\to Y^e$ is a morphism of topological spaces
and $f^* : \mathcal{O}_Y\to f^e_{
*}\mathcal{O}_X$ is a morphism of sheaves such that $f^*_
x : \mathcal{O}_{Y , f (x)}\to\mathcal{O}_{X,x}$ is a local morphism for
any $x\in X^e$. Let $\mathcal{V}$ denote the category of geometric superspaces.

Let $X$ be a geometric superspace. If $U$ is an open subset of $X^e$, then $(U, \mathcal{O}_X|_U)$ is again a geometric superspace, which is called an \emph{open super-subspace} of $X$. In what follows 
$(U, \mathcal{O}_X|_U)$ is denoted just by $U$.

Let $X$ be a geometric superspace. The sheafification of the pre-sheaf $U\to I_{\mathcal{O}_X(U)}=\mathcal{O}_X(U) (\mathcal{O}_X )(U)_1$ is a sheaf of $\mathcal{O}_X$-super-ideals, that is denoted by $\mathcal{I}_X$. The purely even geometric superspace $(X^e, \mathcal{O}_X/\mathcal{I}_X)$ is denoted by $X_{ev}$, and by $X_{res}$, when it is regarded as an geometric space.  

Finally, with each geometric superspace $X$ one can associate a purely even geometric superspace $X_0=(X^e, (\mathcal{O}_X)_0)$, that can be also regarded as a geometric space.

\subsection{Superschemes}

Let $R$ be a super-ring. An \emph{affine superscheme} $SSpec \ R$ can be defined as follows.
The underlying topological space of $SSpec \ R$  coincides
with the prime spectrum of $R$, endowed with the Zariski topology. For any open subset $U\subseteq (SSpec \ R)^e$ the super-ring $\mathcal{O}_{SSpec \ R}(U)$ consists of all locally constant functions $h : U\to\sqcup_{\mathfrak{p}\in U} R_{\mathfrak{p}}$ such that $h(\mathfrak{p})\in R_{\mathfrak{p}} , \mathfrak{p}\in U$. 

For any $f\in R_0$ let $D(f)$ denote the open subset $\{\mathfrak{p}\in (SSpec \ R)^e\mid f\not\in\mathfrak{p}\}$. As in the purely even case, $D(f)$ is isomorphic to $SSpec \ R_f$.

Affine superschemes form a full subcategory of $\mathcal{V}$, which is anti-equivalent to the category of super-rings.

A superspace $X$ is called a \emph{(geometric) superscheme} if there is an open covering $X^e =
\cup_{i\in I} U_i$, such that each open super-subspace $U_i$ is isomorphic to an affine superscheme $SSpec \ R_i$. Superschemes form a full subcategory of $\mathcal{V}$, denoted by $\mathcal{SV}$. 

If $X$ is a superscheme, then any its open super-subspace is a superscheme, called an \emph{open super-subscheme}. A superscheme $Z$ is a \emph{closed super-subscheme} of $X$, if there is a
closed embedding $\iota : Z^e\to X^e$ such that the sheaf $\iota_*\mathcal{O}_Z$ is an epimorphic image of the sheaf $\mathcal{O}_X$. For example, $X_{ev}$ is a closed super-subscheme of $X$.

Finally, a superscheme $Z$ is said to be a \emph{super-subscheme}
of $X$, if $Z$ is isomorphic to a closed super-subscheme of an open super-subscheme of $X$.

It can be easily seen that $Z_{res}$ is an (open, closed) subscheme of $X_{res}$, provided $Z$ is an
(open, closed) super-subscheme of $X$.  

A superscheme $X$ is called \emph{irreducible}, if the topological space $X^e$ is irreducible.
Thus it obviously follows that $X$ is irreducible if and only if $X_{ev}$ is if and only if $X_{res}$ is.

A superscheme $X$ is said to be \emph{Noetherian} if $X$ can be covered by finitely many open affine super-subschemes $SSpec \ R_i$ with $R_i$ to be Noetherian. Note that if a superscheme $X$ is Noetherian, then $X_{res}$ is a Noetherian scheme.

The proof of the following lemma can be copied from the proof of Proposition II.3.2, \cite{hart}.
\begin{lm}\label{Noetheriansuperschemes}
Let $X$ be a superscheme. Then the following are equivalent :
\begin{itemize}
\item[(a)] $X$ is Noetherian; 
\item[(b)] 
\begin{itemize}
\item[(1)] $X^e$ is a quasi-compact topological space;
\item[(2)] for any open affine super-subscheme $U\simeq SSpec \ A$ of $X$, the super-ring $A$ is Noetherian.
\end{itemize}
\end{itemize}
\end{lm}
Let $K$ be a field. A superscheme $X$ is said to be of \emph{finite type} over $K$, if there is a finite open affine covering of $X$ as above, such that each $R_i$ is a finitely generated $K$-superalgebra. 

Similarly to the above lemma one can show that a superscheme $X$ is of finite type over $K$ if and only if $X^e$ is quasi-compact and for any open affine super-subscheme $U\simeq SSpec \ A$ of $X$, $A$ is a finitely generated $K$-superalgebra.

\subsection{Integral superschemes}

A superscheme $X$ is called {\it reduced}, if $X_{res}$ is a reduced scheme. Since $\mathcal{O}_{X_{res}, x}\simeq\overline{\mathcal{O}_x}$, this property is local, i.e. $X$ is reduced if and only if the superring $\mathcal{O}_x$ is reduced for any $x\in X^e$ (cf. \cite{hart}, Exercise II.2.3(a)). 

In particular, an affine superscheme $SSpec \ A$ is reduced if and only if the super-ring $A$ is reduced. Moreover, a superscheme $X$ is reduced if and only if any its open affine super-subscheme is. 

Obviously, if for any open subset $U\subseteq X^e$ the superring $\mathcal{O}_X(U)$ is reduced, then $X$ is reduced. Nevertheless, we do not know whether the converse is true (compare with the definition in \cite{hart}, II, \S 3).

A superscheme $X$ is called {\it integral}, if $X_{res}$ is integral. If additionally each local superring $\mathcal{O}_x$ is a strong superdomain, then $X$ is called {\it strong integral}.
\begin{pr}\label{integral}
A superscheme $X$ is strong integral if and only if the following conditions hold : 
\begin{enumerate}
\item $X$ is irreducible and reduced; 
\item for any its open affine super-subscheme $U$ the superring $\mathcal{O}(U)$ is a strong superdomain.
\end{enumerate}
\end{pr}
\begin{proof}
By Proposition II.3.1, \cite{hart},  $X_{res}$ is integral if and only if Condition (1) holds. 

To prove "only if", assume that $X$ is strong integral. Then Condition (1) holds. Therefore, for any open super-subscheme $U\simeq SSpec \ A$ of $X$, $\overline{A}$ is a domain (cf. \cite{hart}, Example II.3.0.1). 
Furthermore, $A_{\mathfrak{p}}$ is a strong superdomain for any point $\mathfrak{p}\in (SSpec \ A)^e$. This implies that the superdomain $A$ is strong, ensuring Condition (2). Indeed, given $s\in A_0 \setminus A^2_1$, the multiplication $a\mapsto sa, A\to A$ by $s$ is injective
since it is so after localization at every $\mathfrak{p}\in (SSpec \ A)^e$.

Conversely, any local superring $\mathcal{O}_x$ can be identified with a local superring $A_{\mathfrak{p}}$ of an open super-subscheme $U\simeq SSpec \ A$, where $x=\mathfrak{p}\in U^e$.
If $A$ is a strong superdomain, then $A_{\mathfrak{p}}$ is obviously a strong superdomain. This proves the "if" part.
\end{proof}

\subsection{Function superfield}

The following lemma superizes Exercise II.3.6, \cite{hart}.
\begin{lm}\label{genericpoint}
Let $X$ be an integral superscheme and $\xi\in X^e$ be the generic point. Then $\mathcal{O}_{\xi}$ is a superfield that is isomprphic to $SQ(\mathcal{O}_X(U))$ for any open affine super-subscheme $U$ of $X$.
\end{lm}
\begin{proof}
Recall that a point $\xi\in X^e$ is generic if and only if $\overline{\{\xi\}}=X$ if and only if $\xi$ belongs to any (not empty) open subset $V\subseteq X^e$. Let $U$ be an open affine super-subscheme of $X$. Then $A=\mathcal{O}_X(U)$ is a superdomain and the generic point $\xi$ coincides with the smallest prime ideal $I_A$. Thus our lemma obviously follows. 
\end{proof}
Following \cite{hart} we call $\mathcal{O}_{\xi}$ a \emph{function superfield} of $X$ and denote it by $SK(X)$.

\subsection{The functorial approach}

There is an alternative way to define superspaces and superschemes as functors from the category of super-rings/superalgebras to the category of sets. Since we use this approach in Lemma \ref{anopenimmersion} only, we will not introduce this stuff in a complete form. The interested reader can find all necessary notions/definitions in \cite{maszub1}. All we need to note is that the category $\mathcal{SV}$ is equivalent to a full subcategory, $\mathcal{SF}$, of the category of the functors mentioned above (see \cite[Theorem 5.14]{maszub1}).

\subsection{Sheaves of $\mathcal{O}_X$-supermodules}

Let $R$ be a super-ring, and let $M$ be an $R$-supermodule. Analogously to the purely even case, one can define an associated sheaf $\widetilde{M}$ of $\mathcal{O}_X$-supermodules, where $X=SSpec \ R$.
More precisely, for any open subset $U\subseteq (SSpec \ R)^e$ the $\mathcal{O}_X(U)$-supermodule $\widetilde{M}(U)$ consists of all locally constant functions $h : U\to \sqcup_{\mathfrak{p}\in U}M_{\mathfrak{p}}$ such that $h(\mathfrak{p})\in M_{\mathfrak{p}}, \mathfrak{p}\in U$. 

A sheaf $\mathcal{F}$ of $\mathcal{O}_X$-supermodules is called \emph{quasi-coherent}, if $X$ can be covered by open affine super-subschemes $U_i\simeq SSpec \ A_i$, such that for each $i$ there is an $A_i$-supermodule $M_i$ with $\mathcal{F}|_{U_i}\simeq \widetilde{M_i}$ (cf. \cite{hart, zub2}). If, additionally, each supermodule $M_i$ is finitely generated, then the sheaf $\mathcal{F}$ is called \emph{coherent}.
\begin{pr}\label{equivalenceofcategories}
If $X=SSpec \ A$, then the functor $M\mapsto\widetilde{M}$ is an equivalence of the category of $A$-supermodules and the category of quasi-coherent sheaves of $X$-supermodules (both with graded morphisms). Moreover, if $A$ is a Noetherian superalgebra, then this functor is an equivalence of the category of finitely generated $A$-supermodules and the category of coherent sheaves of $X$-supermodules (both with graded morphisms).
\end{pr}
\begin{proof}
By Proposition 2.1, \cite{zub}, this functor is full and faithful. Let $\mathcal{F}$ be a quasi-coherent sheaf of $\mathcal{O}_X$-supermodules. By Proposition 3.1, \cite{zub}, and by Corollary II.5.5, \cite{hart} as well, $\mathcal{F}|_{\mathcal{O}_{X_0}}\simeq \widetilde{M}$, where $M$ is an $A_0$-supermodule. Moreover, if $A$ is Noetherian and $\mathcal{F}$ is coherent, then $A_0$ is also Noetherian and $M$ is a finitele generated $A_0$-(super)module respectively.

Let $f$ denote the natural superscheme morphism $X\to X_0=SSpec \ A_0$, induced by the canonical embedding $A_0\to A$. Then $f^*(\mathcal{F}|_{\mathcal{O}_{X_0}})$ is a sheaf of $\mathcal{O}_X$-supermodules associated with the presheaf
\[U\mapsto \mathcal{O}_X(U)\otimes_{\mathcal{O}_{X_0}(U)}\mathcal{F}(U), U\subseteq X_0^e=X^e.\]
Moreover, the natural morphism $f^*(\mathcal{F}|_{\mathcal{O}_{X_0}})\to \mathcal{F}$ of sheaves of $\mathcal{O}_X$-supermodules, induced by the morphism of presheaves (of $\mathcal{O}_X$-supermodules)
\[\mathcal{O}_X(U)\otimes_{\mathcal{O}_{X_0}(U)}\mathcal{F}(U)\to \mathcal{F}(U),\]
recovers the structure of $\mathcal{F}$ as a sheaf of $\mathcal{O}_X$-supermodules. By Proposition 2.1 (5), \cite{zub}, $f^*(\mathcal{F}|_{\mathcal{O}_{X_0}})\simeq \widetilde{A\otimes_{A_0} M}$, hence the morphism $f^*(\mathcal{F}|_{\mathcal{O}_{X_0}})\to \mathcal{F}$ defines the structure of $A$-supermodule on $M$, such that $\mathcal{F}\simeq \widetilde{M}$. Proposition is proven.
\end{proof}
The following proposition is a superization of Proposition II.5.9, \cite{hart}. 
\begin{pr}\label{closed super-subschemes and (quasi)coherent sheaves of super-ideals}
Let $X$ be a superscheme. For any closed super-subscheme $Y$ the super-ideal sheaf $\mathcal{J}_Y=\ker (\mathcal{O}_X\to i_* \mathcal{O}_Y)$ is quasi-coherent, where $i$ is the corresponding closed embedding $Y^e\to X^e$. If $X$ is Noetherian, then $\mathcal{J}_Y$ is coherent. Conversely, any quasi-coherent super-ideal sheaf on $X$ has the form $\mathcal{J}_Y$ for an uniquely defined closed super-subscheme $Y$ of $X$.
\end{pr}
\begin{proof}
It is easy to see that $Y_0$ is a closed subscheme of $X_0$ with respect to the same closed embedding $i$. Moreover, $i_* \mathcal{O}_Y|_{\mathcal{O}_{X_0}}=i_* (\mathcal{O}_Y|_{\mathcal{O}_{Y_0}})$. By Proposition 3.1, \cite{zub} and Proposition II.5.8(c), \cite{hart},
 $i_* \mathcal{O}_Y|_{\mathcal{O}_{X_0}}$ is a quasi-coherent sheaf of $\mathcal{O}_{X_0}$-supermodules, hence, again by Proposition 3.1, \cite{zub}, it is a quasi-coherent sheaf of $\mathcal{O}_X$-supermodules. Corollary 3.2, \cite{zub}, infers that $\mathcal{J}_Y$ is quasi-coherent.

By Proposition \ref{equivalenceofcategories}, the converse statement is proved just as proving Proposition II.5.9, \cite{hart}. For coherency of $\mathcal{J}_Y$ when $X$ is Noetherian, copy verbatim the proof of Proposition II.5.9, \cite{hart}.
\end{proof}
\begin{cor}
Combining Proposition \ref{equivalenceofcategories} and Proposition \ref{closed super-subschemes and (quasi)coherent sheaves of super-ideals}, one sees that any closed super-subscheme of an affine superscheme
$SSpec \ A$ is isomorphic to $SSpec \ A/I$ for a super-ideal $I$ of $A$. 
\end{cor}
The following lemma superizes Exercise II.5.7, \cite{hart}.
\begin{lm}\label{locallyfreesheaves}
Let $X$ be a Noetherian superscheme, and $\mathcal{F}$ be a coherent $\mathcal{O}_X$-supermodule. Then the following statements hold :
\begin{enumerate}
\item If the stalk $\mathcal{F}_x$ is a free $\mathcal{O}_x$-supermodule for some point $x\in X^e$, then there is a neighborhood $U$ of $x$ such that $\mathcal{F}|_U$ is a free $\mathcal{O}_U$-supermodule of the same rank;
\item $\mathcal{F}$ is a locally free $\mathcal{O}_X$-supermodule if and only if $\mathcal{F}_x$ is a free $\mathcal{O}_x$-supermodule for any $x\in X^e$.
\end{enumerate}
\end{lm}
\begin{proof}
There are an open affine super-subscheme $U\simeq SSpec \ A$ of $X$ and a finitely generated $A$-supermodule $M$, such that $x\in U$ and $\mathcal{F}|_U\simeq\widetilde{M}$ respectively. Then $\mathcal{F}_x\simeq
M_{\mathfrak{p}}$ is a free $A_{\mathfrak{p}}$-supermodule, where $\mathfrak{p}$ is a prime superideal of $A$, that corresponds to the point $x$. In other words, there are (homogeneous) elements $m_1, \ldots , m_t\in M$, which form a basis of the free $A_{\mathfrak{p}}$-supermodule $M_{\mathfrak{p}}$.

Let $N$ denote a free $A$-supermodule with a basis $n_1, \ldots , n_t$, such that the parity of each $n_i$ coincides with the parity of corresponding $m_i$. Let $u : N\to M$ be a morphism of $A$-supermodules, induced by the map $n_i\mapsto m_i, 1\leq i\leq t$. Using Lemma 1.2, \cite{zub2}, and arguing as in \cite{bur}, II, \S 5, Proposition 2, one can easily show that  there is $f\in A_0\setminus\mathfrak{p}_0$ such that $(\mathrm{Coker})_f=(\mathrm{Ker})_f=0$, hence $u_f : N_f\to M_f$ is an isomorphism. In particular, $\widetilde{M}|_{D(f)}\simeq \widetilde{M_f}$ is a free sheaf.

The second statement is now obvious. 
\end{proof}

\section{K\"{a}hler superdifferentials}

\subsection{Supermodules of relative differential forms}

Let $A$ be a superring, $B$ be  an $A$-superalgebra, and $M$ be a left $B$-supermodule. Recall that $M$ is also regarded as an right $B$-supermodule via $mb=(-1)^{|b||m|}bm, b\in B, m\in M$.

An even or odd additive map $d : B\to M$ is called an {\it $A$-superderivation} if 
the following conditions hold :
\begin{enumerate}
\item $d(ab)=(-1)^{|a||d|}a d(b)+d(a)b$;
\item $da=0$ for any $a\in A$.
\end{enumerate}
Let $\mathrm{Der}_A(B, M)$ denote a $\mathbb{Z}_2$-graded abelian group with $\mathrm{Der}_A(B, M)_i$ consisting of all superderivations of parity $i=0, 1$. Observe that $\mathrm{Der}_A(B, M)$ has a natural structure of $B$-supermodule.

The pairs $(M, d)$, where $M$ is a $B$-supermodule and $d\in\mathrm{Der}_A(B, M)$, form a category with (even) morphisms $f : M\to M'$ of $B$-supermodules such that $d'=fd$.

The proof of the following lemma is standard and we leave it for the reader.
\begin{lm}\label{universaldiff} (see \cite{man}, chapter 3, \S 1.8, or \cite{hart}, II.8) There are $B$-supermodules $\Omega_{B/A, ev}$ and $\Omega_{B/A, odd}$, two superderivations $d_0 : B\to\Omega_{B/A, ev}$ and $d_1 : B\to\Omega_{B/A, odd}$, where $|d_0|=0, |d_1|=1$, such that for any $B$-supermodule $M$, compositions with $d_0$ and $d_1$ give isomorphisms (of abelian groups)
\[\mathrm{Der}_A(B, M)_0\simeq\mathrm{Hom}_B(\Omega_{B/A, ev}, M) \ \mbox{and} \ \mathrm{Der}_A(B, M)_1\simeq\mathrm{Hom}_B(\Omega_{B/A, odd}, M)\] respectively. 
\end{lm}
The $B$-supermodules $\Omega_{B/A, ev}$ and $\Omega_{B/A, odd}$ are called the {\it even and odd supermodules of relative differential forms} of $B$ over $A$ respectively. In what follows we denote $\Omega_{B/A, ev}$ just by $\Omega_{B/A}$. 
\begin{rem}\label{finiteness}
\emph{If $B$ is a finitely generated $A$-superalgebra, then $\Omega_{B/A}$ is a finitely generated $B$-supermodule.}
\end{rem}
\begin{rem}\label{anexample}
\emph{Let $B$ be a finitely generated $A$-superalgebra over a Noetherian superalgebra $A$. We have $B\simeq A[X_1, \ldots, X_m\mid Y_1, \ldots, Y_n]/J$, where the super-ideal $J$ is generated by finitely many homogeneous elements, say $f_1, \ldots, f_t$. One can easily show that $\Omega_{B/A}\simeq F/N$, where $F$ is a free $B$-supermodule, freely generated  by the elements $d_0 X_1, \ldots, d_0 X_m, d_0 Y_1, \ldots, d_0 Y_n$, and $N$ is a supersubmodule generated by $d_0 f_1, \ldots, d_0 f_t$.}
\end{rem}
\begin{rem}\label{parityshift}
\emph{Let $\Pi d_0$ denote an odd superderivation $B\to\Pi\Omega_{B/A, ev}$ that takes $b$ to $(-1)^{|b|}d_0 b, b\in B$. Then $(\Pi\Omega_{B/A, ev}, \Pi d_0)\simeq (\Omega_{B/A, odd}, d_1)$ with respect to the isomorphism $\Pi\Omega_{B/A, ev}\simeq \Omega_{B/A, odd}$ of $B$-supermodules, induced by the map 
$d_0 b\mapsto (-1)^{|b|}d_1 b$. Lemma \ref{universaldiff} infers that for any any $B$-supermodule $M$ there is an isomorphism $\mathrm{Der}_A(B, M)_0\to\mathrm{Der}_A(B, M)_1$ of abelian groups, that takes $(M, d)$ to $(\Pi M, \Pi d)$, where
$\Pi d(b)=(-1)^{|b|} db, b\in B$.}
\end{rem}
\begin{pr}\label{anotherdef}
Let $f : B\otimes_A B\to B$ be the diagonal homomorphism $b\otimes b'\mapsto bb', b, b'\in B$, and let $I=\ker f$. Define a map 
$d : B\to I/I^2$ by \[db=1\otimes b-b\otimes 1\pmod{I^2}.\]
Then $(I/I^2, d)\simeq
(\Omega_{B/A, ev}, d_0)$. In particular, $(\Pi (I/I^2), \Pi d)\simeq (\Omega_{B/A, odd}, d_1)$.
\end{pr}
\begin{proof}
The proof of the first statement can be copied from \cite{mats}, Proposition (26.C). The second one follows by Remark \ref{parityshift}.
\end{proof}
\begin{lm}\label{baseextension}(see \cite{mats}, p.184)
If $A'$ and $B$ are $A$-superalgebras, let $B'=B\otimes_A A'$. Then $\Omega_{B'/A'}\simeq\Omega_{B/A}\otimes_B B'$. Furthermore, if $S$ is a multiplicative system in $B_0$, then $\Omega_{S^{-1}B/A}\simeq S^{-1}\Omega_{B/A}$. 
\end{lm}
\begin{proof}
It is easy to see that the map $B'\to\Omega_{B/A}\otimes_B B'$, defined by
\[b\otimes a'\mapsto d_0 b\otimes a', b\in B, a'\in A',\]
is an $A'$-superderivation and it satisfies the property of universality. 

Similarly, the map $S^{-1}B\to S^{-1}\Omega_{B/A}$, defined as 
\[\frac{b}{s}\mapsto\frac{s d_0 b-(d_0 s)b}{s^2}, b\in B, s\in S,\]
is an $A$-superderivation, that also satisfies the property of universality. 
\end{proof}
In propositions below one finds some standard properties of supermodules of relative differential forms. Their proofs can be copied from \cite{mats}, chapter 10,  just verbatim. 
\begin{pr}\label{firstexactsequence}(First Exact Sequence)
Let $A\to B\to C$ be superrings and superring morphisms. Then there is a natural sequence of $C$-supermodules 
\[\Omega_{B/A}\otimes_B C\to\Omega_{C/A}\to \Omega_{C/B}\to 0.\]
Moreover, the map $\Omega_{B/A}\otimes_B C\to\Omega_{C/A}$ has a left inverse if and only if any $A$-superderivation of $B$ into any $C$-supermodule $T$ can be extended to a superderivation of $C$ into $T$.
\end{pr}
\begin{pr}\label{secondexactsequence}(Second Exact Sequence)
Let $B$ be an $A$-superalgebra, let $I$ be a super-ideal of $B$, let $C=B/I$ and $B'=B/I^2$. There is a natural exact sequence of $C$-supermodules
\[I/I^2\to\Omega_{B/A}\otimes_B C\to\Omega_{C/A}\to 0,\]
where the first map takes $b+I^2$ to $d_0 b\otimes 1, b\in I$. Moreover, $\Omega_{B/A}\otimes_B C\simeq\Omega_{B'/A}\otimes_{B'} C$ and the map $I/I^2\to\Omega_{B/A}\otimes_B C$ has a left inverse if and only if the extension
\[0\to I/I^2\to B'\to C\to 0\]
is trivial.
\end{pr}

\subsection{Sheaves of K\"{a}hler superdifferentials}

\begin{lm}\label{anopenimmersion}
Let $f : X\to Y$ be a superscheme morphism. For any open affine super-subschemes $U\subseteq X$ and $V\subseteq Y$ such that $f(U)\subseteq V$, the induced morphism $U\times_V U\to X\times_Y X$ is an isomorphism onto an open super-subscheme of $X\times_Y X$.
\end{lm}
\begin{proof}
By the remark after Proposition 5.12, \cite{maszub1}, one needs to check the analogous statement in the category $\mathcal{SF}$, which is obvious (see \cite{jan}, I.1.7(3)).
\end{proof}
For the above morphism, let $\Delta$ denote the diagonal morphism $X\to X\times_Y X$. 
Arguing as in \cite[\S 8]{hart} and using Lemma \ref{anopenimmersion}, one can show that $\Delta$ is an isomorphism of $X$ onto a closed super-subscheme $\Delta(X)$ of an open super-subscheme $W$ of $X\times_Y X$, which is defined by a sheaf of super-ideals $\mathcal{J}\subseteq \mathcal{O}_W$.

Following \cite[\S 8]{hart}, we define the \emph{sheaf of K\"{a}hler superdifferentials} of $X$ over $Y$ to be the $\mathcal{O}_X$-supermodule $\Omega_{X/Y}=\Delta^*(\mathcal{J}/\mathcal{J}^2)$. The isomorphism $\Delta$ induces
\[\mathcal{O}_{\Delta(X)}\simeq\mathcal{O}_{\Delta(X)}\simeq\mathcal{O}_{W}/\mathcal{J},\]
through which we identify first $\mathcal{O}_X$ with $\mathcal{O}_{W}/\mathcal{J}$, and then $\Omega_{X/Y}$ with $\mathcal{O}_{W}/\mathcal{J}$-supermodule $\mathcal{J}/\mathcal{J}^2$.

More precisely, let $U=SSpec \ B$ be an affine open super-subscheme of $X$ and $V=SSpec \ A$ be an affine open super-subscheme of $Y$ such that $f(U)\subseteq V$, then $U\times_V U\simeq SSpec \ (B\otimes_A B)$ and $\Delta(X)\cap (U\times_V U)$ is a closed super-subscheme defined by the kernel of the diagonal homomorphism $B\otimes_A B\to B$. Proposition \ref{anotherdef} implies that $\Omega_{U/V}\simeq\widetilde{\Omega_{B/A}}$ and by covering $X$ and $Y$ with $U$ and $V$ as above, one can define $\Omega_{X/Y}$  by gluing the corresponding sheaves $\widetilde{\Omega_{B/A}}$ (cf. \cite{hart}, Remark II.8.9.2). In particular, the $\mathcal{O}_X$-supermodule $\Omega_{X/Y}$ is quasi-coherent.
Moreover, if $X$ is Noetherian and $f$ is a morphism of finite type, then $\Omega_{X/Y}$ is coherent.
Finally, gluing superderivations $d_0 : B\to\Omega_{B/A}$ one can construct a morphism $\mathcal{O}_X\to\Omega_{X/Y}$ of sheaves of superspaces, which is a superderivation of the local superrings at any point.

Again, as in \cite[\S 8]{hart} we formulate sheaf counterparts of the algebraic results of the previous subsection. Their proofs are standard and we leave them for the reader. 
\begin{pr}\label{baseextensionforsheaves}
If $f' : X'=X\times_Y Y'\to Y'$ is a base extension of $f : X\to Y$ with $g : Y'\to Y$, then $\Omega_{X'/Y'}\simeq g'^*(\Omega_{X/Y})$, where $g' : X'\to X$ is the first projection.  
\end{pr}
\begin{pr}\label{firstexactsequenceforsheaves}
Let $f : X\to Y$ and $g : Y\to Z$ be superscheme morphisms. Then there is an exact sequence of $\mathcal{O}_X$-supermodules
\[f^*\Omega_{Y/Z}\to\Omega_{X/Z}\to\Omega_{X/Y}\to 0.\]
\end{pr}
\begin{pr}\label{secondexactsequenceforsheaves}
Let $f : X\to Y$ be a superscheme morphism, and let $Z$ be a closed super-subscheme of $X$, defined by a superideal sheaf $\mathcal{J}$. There is an exact sequence of $\mathcal{O}_Z$-supermodules
\[\mathcal{J}/\mathcal{J}^2\to \Omega_{X/Y}\otimes_{\mathcal{O}_X}\mathcal{O}_Z\to\Omega_{Z/Y}\to 0.\] 
\end{pr}

\section{Krull super-dimension}

\subsection{Odd parameters}

Let $R$ be a Noetherian super-ring. Assume that the Krull dimension $\mathrm{Kdim}(R_0)$ of $R_0$ is finite. 
Let $y_1,\dots, y_s$ be a sequence of odd elements in $R_1$. For any subset $I$ of the set $\underline{s}=\{1, 2, \ldots, s\}$ we let 
\[ y^I = \prod_{i\in I} y_i \]
denote the product in $R$. This product can change only by sign, according to the
order of the consisting elements. We may not and we will not refer to the order to
discuss the product. Let
\[ \mathrm{Ann}_{R_0}(y^I) = \{ r \in R_0 \mid r y^I=0\} \]
denote the ideal of $R_0$ consisting of those elements which annihilate $y^I$. 

We say that $y_1, \dots, y_s$ form a \emph{system of odd parameters} if
\[ \mathrm{Kdim}(R_0)=\mathrm{Kdim}(R_0/\mathrm{Ann}_{R_0}(y^{\underline{s}})). \]
In other words, the elements $y_1, \ldots, y_s$ form a system of odd parameters of $R$ if and only if
there is a longest prime chain of $R_0$, say $\mathfrak{p}_0\subseteq \ldots\subseteq\mathfrak{p}_n, n=\mathrm{Kdim}(R_0)$, such that $\mathrm{Ann}_{R_0}(y^{\underline{s}})\subseteq \mathfrak{p}_0$.

Since for any $I\subseteq\underline{s}$ the ideal $\mathrm{Ann}_{R_0}(y^I)$ is contained in $\mathrm{Ann}_{R_0}(y^{\underline{s}})$, the elements $y_i, i\in I$, form a system of odd parameters whenever $y_1, \ldots, y_s$ do.

Let $r= \mathrm{Kdim}(R_0)$. Let $s$ be the largest
number of those elements in $R_1$ which form a system of odd parameters. 
The \emph{Krull super-dimension} $\mathrm{Ksdim}(R)$ of $R$ is defined by
\[ \mathrm{Ksdim}(R)= r \mid s . \]
Moreover, the Krull dimension of $R_0$, that is $r$, is called the \emph{even Krull dimension} of $R$, and $s$ is called the \emph{odd Krull dimension} of $R$. They are denoted by $\mathrm{Ksdim}_0(R)$ and $\mathrm{Ksdim}_1(R)$, respectively. 

Finally, $\mathrm{Ksdim}_1(R)=0$ if and only if for any $y\in R_1$ and for any prime chain $\mathfrak{p}_0\subseteq \ldots\subseteq\mathfrak{p}_n$ in $R_0$ of length $n=\mathrm{Kdim}(R_0)$ we have $\mathrm{Ann}_{R_0}(y)\not\subseteq\mathfrak{p}_0$. Moreover, since $R_1$ is a finitely generated $R_0$-module, the latter is equivalent to $\mathrm{Ann}_{R_0}(R_1)\not\subseteq\mathfrak{p}_0$ for any prime $\mathfrak{p}_0$ as above.

Let $B$ be a commutative ring and $M$ be a $B$-module. Recall that $\dim(M)$ denotes the dimension of $M$ as a $B$-module, i.e.
\[\dim(B)=\mathrm{Kdim}(B/\mathrm{Ann}_B(M)).\] 
\begin{pr}\label{another definition of odd super-dimension}
Let $R$ be a Noetherian super-ring with $\mathrm{Kdim}(R_0)<\infty$. Choose a system of generators of $R_0$-module $R_1$, say $y_1, \ldots, y_d$.
Then the following conditions are equivalent :
\begin{enumerate}
\item There is a system of odd parameters of $R$ of cardinality $l\geq 1$;
\item For some $1\leq i_1<\ldots <i_l\leq d$ the elements $y_{i_1}, \ldots , y_{i_l}$ form a system of odd parameters;
\item $\dim(R_1^l)=\mathrm{Kdim}(R_0)$.
\end{enumerate}
In particular, we have 
\[\mathrm{Ksdim}_1(R)=\max\{l\mid \dim(R_1^l)=\mathrm{Kdim}(R_0)\}.\]
\end{pr}
\begin{proof}
For a given $l\leq\mathrm{Ksdim}_1(R)$, let $z_1, \ldots, z_l$ be a system of odd parameters. Then
$\mathrm{Ann}_{R_0}(R_1^l)\subseteq \mathrm{Ann}_{R_0}(z^{\underline{l}})$ implies
$\dim(R_1^l)=\mathrm{Kdim}(R_0)$. Conversely, if $\dim(R_1^l)=\mathrm{Kdim}(R_0)$, then there is a longest prime chain $\mathfrak{p}_0\subseteq\ldots\subseteq\mathfrak{p}_r$ in $R_0$, where $r=\mathrm{Kdim}(R_0)$, such that
\[\cap_{I\subseteq\underline{d}, |I|=l}\mathrm{Ann}_{R_0}(y^I)=\mathrm{Ann}_{R_0}(R_1^l)\subseteq\mathfrak{p}_0,\]
and therefore $\mathrm{Ann}_{R_0}(y^I)\subseteq\mathfrak{p}_0$ for some $I$. 
Proposition is proven.
\end{proof}

\subsection{Regular sequences}

Recall that an odd element $y\in R$ is called \emph{odd regular}, if $\mathrm{Ann}_R(y)=Ry$ (cf. \cite[p.67]{sm}). Besides, a sequence $y_1, \ldots, y_t$ of odd elements of $R$ is said to be \emph{odd regular} if for each $i$ the element $y_i$ is odd regular modulo $Ry_1+\ldots+Ry_{i-1}$. By \cite[Corollary 3.1.2]{sm} the sequence $y_1, \ldots, y_t$ is odd regular if and only if $\mathrm{Ann}_R(y^{\underline{t}})=Ry_1+\ldots +Ry_t$. Thus any odd regular sequence form a system of odd parameters. 
\begin{lm}\label{whensuper-dimensionisdecreasing}
Let $R$ be a Noetherian super-ring of Krull super-dimension $r|s$. Let $y\in R_1$. The following hold :
\begin{itemize}
\item[(a)] If $y$ is contained in a system of odd parameters of the largest length $s$, then $\mathrm{Ksdim}_1(R/Ry)\geq s-1$.
\item[(b)] If $y_1, \ldots, y_t$ is an odd regular sequence, then $t\leq s$ and 
\[\mathrm{Ksdim}(R/(Ry_1+\ldots Ry_t))=r|(s-t).\]
\item[(c)] Any odd regular sequence can be extended to a system of odd parameters of the largest length $s$. 
\end{itemize}
\end{lm}
\begin{proof}
(a) Let $y_1, \ldots , y_s$ be a system of odd parameters such that $y_s = y$. Then
$y_1, \ldots , y_{s-1}$ form a system of odd parameters modulo $Ry$, which proves (a). Indeed, there obviously holds
\[\mathrm{Ann}_{R_0/R_1 y}(y^{\underline{(s-1)}}\mathrm{mod}(Ry))\subseteq\mathrm{Ann}_{R_0}(y^{\underline{s}})/R_1 y,\]
and the latter is obviously included in the first of
\[\mathfrak{p}_0/R_1 y\subseteq \ldots\subseteq\mathfrak{p}_r/R_1 y,\]
where $(Ry_1\subseteq) \mathfrak{p}_0\subseteq \ldots\subseteq\mathfrak{p}_r$ is some longest prime chain of $R_0$.

(b) To prove this when $t=1$, assume that $y$ is odd regular. Obviously we have $\mathrm{Ksdim}_0(R/Ry)=r$ and $s\geq 1$. If $\mathrm{Ksdim}_1(R/Ry)\geq s$, then there is a system of odd parameters of length $s$ modulo $Ry$, say $y_1, \ldots, y_s$. Since $ry^{\underline{s}}y=0$ is equivalent to $ry^{\underline{s}}\in Ry$, $\mathrm{Ann}_{R_0/R_1 y}(y^{\underline{s}}Ry)$ coincides with
$\mathrm{Ann}_{R_0}(y^{\underline{s}}y)/R_1 y$. Therefore, $y_1, \ldots, y_s, y$ form a system of odd parameters. This contradiction, combined with (a), proves the result when $t=1$. For the general case use the obtained result repeatedly.

(c) Assume that $y_1, \ldots, y_t$ is an odd regular sequence. Then we have $t\leq s$ by
(b). Moreover, odd elements $y_{t+1}, \ldots , y_s$ can be chosen so that they form, modulo
$Ry_1 + \ldots + Ry_t$, a system of odd parameters. We see that $y_1, \ldots , y_t, \ldots , y_s$ is a desired system of odd parameters.
\end{proof}

\subsection{Noether Normalization Theorem for superalgebras}

Let $K$ be a field. Suppose that $A$ is a finitely generated $K$-superalgebra. Then the $K$-algebra
$A_0$ is finitely generated. By the Noether Normalization Theorem \cite[(14G)]{mats}, $A_0$ contains
a polynomial subalgebra $B=k[X_1,\dots, X_r]$ over which $A_0$ is integral, and so $r = \mathrm{Ksdim}_0(A)$. 

The following proposition shows that odd parameters are nothing else but elements {\bf algebraically independent} over a subalgebra generated by "even" parameters. Geometrically, this means that the even and odd components of the Krull super-dimension of $A$ are equal to the number of even and odd {\bf degrees of freedom} of the superscheme $SSpec \ A$. More precisely, if $\mathrm{Ksdim}(A)=r\mid s$, then there is a certain (finite) superscheme morphism 
\[SSpec \ A \to A^{r|s}=SSpec \ K[X_1, \ldots X_r\mid Y_1, \ldots , Y_s].\]
\begin{pr}\label{NNT}
For a sequence of odd elements $y_1,\dots, y_s$ the following are equivalent:
\begin{itemize}
\item[(a)] $y_1,\dots, y_s$ form a system of odd parameters;
\item[(b)] For some/any polynomial subalgebra $B \subset A_0$ as above, 
$\mathrm{Ann}_B(y^{\underline{s}}):=B \cap \mathrm{Ann}_{A_0}(y^{\underline{s}})$ equals $0$;
\item[(c)] For some/any polynomial subalgebra $B \subset A_0$ as above, the $B$-superalgebra map
\[ \nu : B[Y_1,\dots,Y_s] \to A \]
which is defined on the polynomial $B$-superalgebra in $s$ odd variables by $\nu(Y_i)=y_i$, $1 \le i \le s$,
is an injection. 
\end{itemize}
In particular, $\mathrm{Ksdim}_1(A)=0$ if and only if $\mathrm{Ann}_{B}(A_1)\neq 0$.
\end{pr}
\begin{proof}
There holds (a) $\Leftrightarrow$ (b) $\Leftarrow$ (c). Indeed, note
from \cite[Theorem 20]{mats} that $\mathrm{Kdim}(A_0/\mathrm{Ann}_{A_0}(y^{\underline{s}}))=
\mathrm{Kdim}(B/\mathrm{Ann}_B(y^{\underline{s}}))$. Since $B$ is an integral domain, any its non-zero ideal has a hight at least $1$. Thus 
\[\mathrm{Kdim}(B/\mathrm{Ann}_B(y^{\underline{s}}))\leq \mathrm{Kdim}(B)-\mathrm{ht}(\mathrm{Ann}_B(y^{\underline{s}})) < \mathrm{Kdim}(B)\]
if and only if $\mathrm{Ann}_B(y^{\underline{s}})\neq 0$. 

For (b) $\Rightarrow$ (c), we wish to prove,
assuming (b), that 
\[ \sum_{I, I\subseteq\underline{s}}a_I\, y^I = 0,
\quad a_I\in B, \]
implies that all the coefficients $a_I$ equal $0$. Assume the contrary. Choose $a_I \neq 0$ with minimal $|I|$. Then, multiplying by the product $y^{\underline{s}\setminus I}$, one obtains $a_I y^{\underline{s}}=0$, a contradiction. 

Since $A_1$ is a finitely generated $B$-module, the last statement is now obvious. 
\end{proof}

Choose elements $y_1, \ldots , y_n\in A_1$, which form a system of generators of the $B$-module $I_A/I_A^2\simeq A_1/A_1^3$. Then the $B$-module $I_A$ is generated by the elements $y^I=y_{i_1}\ldots y_{i_k}$, where $I=\{i_1<\ldots < i_k\}$ runs over all subsets of $\underline{n}=\{1, 2, \ldots , n\}$.

For each $I$ let $J_I$ denote the super-ideal $Ann_B(y^I)$. These super-ideals are partially ordered by $J_I\subseteq J_{I'}$ whenever $I\subseteq I'$. Define a set $\Gamma=\{ I\subseteq\underline{n}\mid J_I=0\}$. By the above, for any $I\in\Gamma$ the inclusion $I'\subseteq I$ infers $I'\in\Gamma$. Let $k$ denote $\max\{|I|\mid I\in\Gamma\}$, where $|I|$ denotes the cardinality of $I$.
\begin{lm}\label{odd-Krull} We have $k=\mathrm{Ksdim}_1(A)$.
\end{lm}
\begin{proof}
Choose a set $I\in\Gamma$ of maximal cardinality $k$. Proposition \ref{NNT} (b) implies that the elements $y_i, i\in I$, form a system of odd parameters. Assume that there is a system of odd parameters of cardinality $k+1$, say $z_1, \ldots , z_{k+1}$. Then 
\[z_j=\sum_{I\subseteq\underline{n}} b^{(j)}_I y^I,b^{(j)}_I\in B, 1\leq j\leq k+1,\] 
and the product $z^{\underline{k+1}}$ is equal to
\[\sum_{I\subseteq\underline{n}}(\sum_{I_1\sqcup\ldots\sqcup I_{k+1}=I}\pm b^{(1)}_{I_1}\ldots b^{(k+1)}_{I_{k+1}})y^I .\]
Since the cardinality of each $I$ from the above sum is at least $k+1$, we have $0\neq \prod_{|I|\geq k+1}J_{I}\subseteq Ann_B(z^{\underline{k+1}})$, hence $z_1, \ldots , z_{k+1}$ do not form a system of odd parameters. This contradiction concludes the proof.
\end{proof}
It is well known that the Krull dimension of a factor-ring $A/I$ is at most the Krull dimension of a ring $A$. Surprisingly, the odd Krull dimension of a quotient of a super-ring $A$ can be greater than the odd Krull dimension of $A$. 

Recall that a polynomial $K$-superalgebra $K[X_1, \ldots, X_m\mid Y_1, \ldots, Y_n]$ is denoted, briefly, by $K[X\mid Y]$.
\begin{example}\label{strangeodd}
\emph{Let us consider a superalgebra $A=K[X\mid Y]/J$, where
\[J=\sum_{I\subseteq\underline{n}, I\cap\underline{l}\neq\emptyset}K[X]X_1 Y^I ,\]
and $n\geq l\geq 1, m\geq 1$. Then $B=K[X]$ is isomorphically mapped onto a subalgebra of $A$, over which $A_0$ is a finite module. Moreover, the residue classes of $Y_1, \ldots , Y_n$ form a system of generators of a $B$-module $I_A/I_A^2$.} 

\emph{Since $\mathrm{Ann}_B(Y^I)\neq 0$ if and only if $I\cap\underline{l}\neq\emptyset$, Lemma \ref{odd-Krull} implies that the residue classes of $Y_{l+1}, \ldots , Y_n$ form a system of odd parameters of the largest cardinality in $A$, i.e. $\mathrm{Ksdim}( A)=(m| n-l)$.}

\emph{On the other hand, the superalgebra 
\[C=K[X\mid Y]/K[X\mid Y]X_1\simeq K[X_2, \ldots , X_m\mid Y]\]
is a quotient of $A$. Moreover, $\mathrm{Ksdim}( C)=(m-1\mid n)$, hence $\mathrm{Ksdim}_1(C) > \mathrm{Ksdim}_1(A)$!}
\end{example}
\begin{lm}\label{whensuperdimensiondoesnotchange}
Let $A$ be a finitely generated $K$-superalgebra, and $S$ be a multiplicative subset of $A_0$. If $\mathrm{nil}(A)$ is a prime super-ideal, then $\mathrm{Ksdim}_1(A)=\mathrm{Ksdim}_1(S^{-1}A)$.
\end{lm}
\begin{proof}
Since $\mathrm{nil}(A_0)$ is a smallest nilpotent prime ideal of $A_0$, $S^{-1}\mathrm{nil}(A_0)$ is a smallest nilpotent prime ideal of $S^{-1}A_0$ as well. Therefore, the elements $\frac{y_1}{s_1}, \ldots, \frac{y_t}{s_t}\in S^{-1}A_1$ form a system of odd parameters if and only if $\mathrm{Ann}_{S^{-1}B_0}(\frac{y_1}{s_1}\ldots\frac{y_t}{s_t})\subseteq S^{-1}\mathrm{nil}(A_0)$. Using $\mathrm{Ann}_{S^{-1}A_0}(\frac{y_1}{s_1}\ldots\frac{y_t}{s_t})=S^{-1}\mathrm{Ann}_{A_0}(y_1\ldots y_t)$, one obtains the required equality. 
\end{proof}
\subsection{One relation superalgebras}

Let $A$ denote the polynomial superalgebra $K[X|Y]$ in $r$ even and $s$ odd free generators. For a nonzero homogeneous element 
\[f=\sum_{L\subseteq\underline{s}}c_L Y^L, c_L\in K[X],\]
let $B$ denote an \emph{one relation superalgebra} $A/Af$. 

Below we discuss the problem of calculating of
$\mathrm{Ksdim}_1(B)$. To simplify our calculations we suppose that $c_{\emptyset}=0$, that is $f\in I_A$.

Since $\overline{B}\simeq\overline{A}\simeq K[X]$, Lemma \ref{whensuperdimensiondoesnotchange} infers that 
$\mathrm{Ksdim}_1(B)=\mathrm{Ksdim}_1((K[X]\setminus 0)^{-1}B)$. In other words, without loss of generality one can assume that $A=K[Y]$.
Further, by Lemma \ref{odd-Krull} the odd Krull dimension of $B$ is equal to 
\[\max\{|L|\mid Y^L\not\in Af\}.\]

The set $\mathrm{Exp}(f)=\{L\mid c_L\neq 0\}$ is partially ordered by inclusion. Let $L_1, \ldots , L_t$ be all pairwise different minimal elements of $\mathrm{Exp}(f)$. The collection $L_1, \ldots, L_t$ is called a \emph{basement} of $f$. 

Observe that if $Y^L$ belongs to $Af$, then there is $1\leq i\leq t$, such that $L_i\subseteq L$. Therefore, if $L$ does not contain any $L_i$, then $Y^L\not\in Af$. 

Since $Y^{\underline{s}}$ is equal to $Y^{\underline{s}\setminus L_1}f$ (up to a nonzero scalar multiple), there always holds $\mathrm{Ksdim}_1(B)\leq s-1$.

A subset $K\subseteq\underline{s}$ of minimal cardinality, which meets each $L_i$, is said to be the \emph{extremal set} of $f$. If $K$ is an extremal set of $f$, then $k=|K|$ is called the \emph{index} of $f$, and it is denoted by $\mathrm{ind}(f)$. 
For example, $\mathrm{ind}(f)=1$ if and only $\cap_{1\leq i\leq t}L_i\neq\emptyset$. Furthermore, $\mathrm{ind}(f)\leq t$ and $\mathrm{ind}(f)=t$ if and only if  $L_i\cap L_j=\emptyset$ for any $1\leq i\neq j\leq t$.
\begin{lm}\label{lowerbound}
We have $\mathrm{Ksdim}_1(B)\geq s-\mathrm{ind}(f)$. 
\end{lm}
\begin{proof}
Let $K$ be an extremal set of $f$. One easily sees that $Y^{\underline{s}\setminus K}\not\in Af$. 
\end{proof}
Lemma \ref{lowerbound} implies that the worst lower bound for $\mathrm{Ksdim}_1(B)$ is $s-t$. The following example shows this estimate is achievable.
\begin{ex}\label{worstestimate}
\emph{Let $\mathrm{ind}(f)=t$ and each $L_i$ has cardinality at least $t$. Then $\mathrm{Ksdim}_1(B)=s-t$. In fact, if $|L|\geq s-t+1$, then $\underline{s}\setminus L$ is not extremal, whence $L_i\subseteq L$. Moreover, any $L_j, j\neq i,$ meets $L\setminus L_i$. Thus $\pm c_{L_i}Y^L=Y^{L\setminus L_i}f$.} 
\end{ex}
\begin{lm}\label{aform}
For any $1\leq i\leq t$ there is an element $g=Y^{L_i}+h$, such that the following conditions hold :
\begin{itemize}
\item[(a)] the minimal elements of $\mathrm{Exp}(h)$ are exactly $L_j, j\neq i$; 
\item[(b)] for any nonzero term $d_L Y^L$ of $h$, the exponent $L$ does not contain $L_i$;
\item[(c)] $Af=Ag$.
\end{itemize}
\end{lm}
\begin{proof}
The polynomial $f$ can be represented as $f=Y^{L_i} p + h'$, where $p$ is an invertible element of $A$, such that the exponent of any its nonzero term does not meet $L_i$. Furthermore, the minimal elements of $\mathrm{Exp}(h')$ are exactly $L_j, j\neq i$, and the exponent of any its nonzero term does not contain $L_i$. It is now obvious that $g=p^{-1}f$ is the required polynomial.
\end{proof}
An element $g$ as in the above lemma is called a \emph{form of $f$ reduced in $L_i$}.
\begin{cor}\label{ifonesetissingleton}
If some $L_i$ is a singleton, then $\mathrm{Ksdim}_1(B)=s-1$.
\end{cor}
\begin{proof}
Let $g=Y_j +h$ be a form of $f$ reduced in $L_i=\{j\}$. The map 
\[Y_j\mapsto g, Y_k\mapsto Y_k, k\neq j,\]
induces an automorphism of superalgebra $A$, which takes the odd regular element $Y_j$ to the odd regular element $g$. Lemma \ref{aform}(c) and Lemma \ref{whensuper-dimensionisdecreasing}(b) imply
$\mathrm{Ksdim}_1(B)=s-1$.
\end{proof}
The following proposition shows that the lower bound in Lemma \ref{lowerbound} is not always sharp.
\begin{pr}\label{t=2}
If $t=2$, then $\mathrm{Ksdim}_1(B)=s-2$ if and only if $L_1\cap L_2=\emptyset$ and both $|L_1|$ and $|L_2|$ are at least $2$, otherwise $\mathrm{Ksdim}_1(B)=s-1$.
\end{pr}
\begin{proof}
By Lemma \ref{lowerbound} we have $\mathrm{Ksdim}_1(B)\geq s-2$. Moreover, if $\mathrm{Ksdim}_1(B)=s-2$, then $L_1\cap L_2=\emptyset$. Besides, Corollary \ref{ifonesetissingleton} infers $|L_1|, |L_2|\geq 2$.

Conversely, assume that $L_1\cap L_2=\emptyset$ and both $L_1$ and $L_2$ have cardinalities at least two. Then any $1\leq i\leq s$ does not belong either to $L_1$, or to $L_2$. In both cases $Y^{\underline{s}\setminus i}\in Af$. For example, if $i\not\in L_1$, then $(\underline{s}\setminus (L_1\sqcup i))\cap L_2\neq\emptyset$ and therefore, $Y^{\underline{s}\setminus (L_1\sqcup i)} f$ is equal to
\[\pm c_{L_1}Y^{\underline{s}\setminus i}\pm c_{L_1\sqcup i}Y^{\underline{s}}.\]
Since $Y^{\underline{s}}$ belongs to $Af$, so $Y^{\underline{s}\setminus i}$ does.
\end{proof}
We do not know whether $\mathrm{Ksdim}_1(B)$ is completely defined by the basement of $f$ for any $t\geq 3$, similarly to the above proposition. 

\section{Regular super-rings}

From now on all super-rings are supposed to be Noetherian, unless otherwise stated.

Let $A$ be a local super-ring with a maximal super-ideal $\mathfrak{m}$. Let $K$ denote its residue
field $A/\mathfrak{m}$. Observe also that $\mathrm{Kdim}(A_0)=r<\infty$ (cf. \cite[Corollary 11.11]{AM}). Set $\mathrm{Ksdim}(A)=r|s$.

Note that
\begin{equation}\label{eq:m/m^2}
\mathfrak{m}/\mathfrak{m}^2= 
(\mathfrak{m}_0/\mathfrak{m}_0^2+A_1^2)\oplus (A_1/\mathfrak{m}_0A_1).
\end{equation}
This is a super-vector space over $K$, which is finite-dimensional since $\mathfrak{m}$ is finitely
generated as an $A_0$-module. 

\begin{lm}\label{lem:odd_generators}
If $y_1,\dots, y_s$ form a system of odd parameters consisting of $s$ elements, we have
\[ s \le \dim_K((\mathfrak{m}/\mathfrak{m}^2)_1). \]
\end{lm}
\begin{proof}
Suppose $t= \dim_K((\mathfrak{m}/\mathfrak{m}^2)_1)$. Let $z_1,\dots, z_t$ be elements
of $A_1$ which give rise to a $K$-basis in $(\mathfrak{m}/\mathfrak{m}^2)_1=A_1/\mathfrak{m}_0A_1$. 
By Nakayama's Lemma this is equivalent to saying that $z_1,\dots, z_t$ form a minimal system of
generators of the $A_0$-module $A_1$. Therefore, each $y_i$ is
an $A_0$-linear combination of them. If $s > t$, it follows that $y^{\underline{s}}=0$, whence 
$y_1,\dots, y_s$ cannot form a system of odd parameters. 
\end{proof}
Lemma \ref{lem:odd_generators} shows 
\[ \mathrm{Ksdim}(A) \le \mathrm{sdim}_K(\mathfrak{m}/\mathfrak{m}^2), \]
or namely, $r \le \mathrm{dim}_K((\mathfrak{m}/\mathfrak{m}^2)_0)$, 
$s \le \mathrm{dim}_K((\mathfrak{m}/\mathfrak{m}^2)_1)$.
We say that $A$ is \emph{regular} if $\mathrm{Ksdim}(A) = \mathrm{sdim}_K(\mathfrak{m}/\mathfrak{m}^2)$.

\begin{pr}\label{prop:Ksdim_regular}
For a local super-ring $(A, \mathfrak{m})$, the following are equivalent:
\begin{itemize}
\item[(a)] $A$ is regular;
\item[(b)] 
\begin{itemize}
\item[(i)] $\overline{A}$ is regular, and
\item[(ii)] for some/any minimal system $z_1,\dots, z_s $ of generators of the $A_0$-module $A_1$, we have 
$\mathrm{Ann}_{A_0}(z^{\underline{s}}) = A_1^2$.
\end{itemize} 
\end{itemize} 
\end{pr}
\begin{proof}
(a) $\Rightarrow$ (b).\ Assume (a). Since 
$(\mathfrak{m}/\mathfrak{m}^2)_0 = \overline{\mathfrak{m}}/\overline{\mathfrak{m}}^2$ and
$\mathrm{Kdim}(A_0)= \mathrm{Kdim}(\overline{A})$,
we have (i) of (b). 
Let $s= \dim_K((\mathfrak{m}/\mathfrak{m}^2)_1)$. Then we have a system $y_1,\dots, y_s$ of
odd parameters. Since $\overline{A} =A_0/A_1^2$ is an integral domain (see \cite[Theorem 48]{mats}), arguing as in the proof of Proposition \ref{NNT} we obtain
\[ \mathrm{Ann}_{A_0}(y^{\underline{s}}) \subset  A_1^2. \]
Given an arbitrary minimal system of generators as in (ii), 
the presentation of each $y_i$ as an $A_0$-linear combination of 
$z_1,\dots, z_s$ gives $y^{\underline{s}}=a z^{\underline{s}}$ for some $a \in A_0$, and so
\[ \mathrm{Ann}_{A_0}(z^{\underline{s}}) \subset \mathrm{Ann}_{A_0}(y^{\underline{s}}). \]
Since one sees $A_1 z^{\underline{s}}=0$, and so
\[ A_1^2 \subset \mathrm{Ann}_{A_0}(z^{\underline{s}}), \]
it follows that $\mathrm{Ann}_{A_0}(z^{\underline{s}}) = A_1^2$.

(b) $\Rightarrow$ (a).\ This is now easy. Note that any minimal system of generators as in (ii)
of (b) form a system of odd parameters. 
\end{proof}

\subsection{Regular local super-rings}\label{sec:regular_locall}

Let $A$ be a (not necessary local) Noetherian super-ring. Recall that $I_A= AA_1$.
This $I_A$ is nilpotent since $A_1$ is now finitely generated as an $A_0$-module. 
Given a positive integer $n$, we have
\[ I_A^n= \begin{cases} A_1^{n+1} \oplus A_1^n & n~~\text{odd},\\
A_1^n \oplus A_1^{n+1} & n~~\text{even}, \end{cases} \]
whence 
$I_A^n/I_A^{n+1}= A_1^n/A_1^{n+2}$; this is an $\overline{A}$-module, which is finitely
generated since $A_1^n$ is finitely generated as an $A_0$-module. 
We define
\[ \mathsf{gr}_{I_A}(A) := \bigoplus_{n\ge 0} I_A^n/I_A^{n+1} =\overline{A} \oplus I_A/I_A^2 \oplus \dots. \]
This is a graded superalgebra over $\overline{A}$. We let
\[ \lambda_A : \wedge_{\overline{A}}(I_A/I_A^2) \to \mathsf{gr}_{I_A}(A) \]
denote the graded $\overline{A}$-superalgebra map induced by the embedding $I_A/I_A^2 \to \mathsf{gr}_{I_A}(A)$.
One sees that this is a surjection. If the $A_0$-module $A_1$ is generated by $s$ elements, then $A_1^{s+1}=0$,
whence $\wedge_{\overline{A}}^n(I_A/I_A^2) = 0= \mathsf{gr}_{I_A}(A)(n)$ for all $n > s$. 

In the remaining of this subsection we suppose that $(A, \mathfrak{m})$ is a local super-ring with 
residue field $K=A/\mathfrak{m}$. 
We let $\overline{\mathfrak{m}}=\mathfrak{m}_0/A_1^2$ denote the maximal ideal
of $\overline{A}=A_0/A_1^2$. 

\begin{pr}\label{prop:regular_local}
For a local super-ring $A$, the following are equivalent:
\begin{itemize}
\item[(a)] $A$ is regular;
\item[(b)]
\begin{itemize} \item[(i)] The ring $\overline{A}$ is regular, 
\item[(ii)] the $\overline{A}$-module $I_A/I_A^2$ is free (of rank $\mathrm{Ksdim}_1(A)$), and 
\item[(iii)] $\lambda_A$ is an isomorphism.
\end{itemize}
\end{itemize}
\end{pr}
\begin{proof}
Assume (a). Then we have (i) and (ii) of (b) in Proposition \ref{prop:Ksdim_regular};
(i) is the same as (i) of (b) above. If we choose arbitrarily a minimal system 
$y_1,\dots,y_s$ of generators of the
$A_0$-module $A_1$, then 
\begin{equation}\label{eq:Ann}
\mathrm{Ann}_{A_0}(y^{\underline{s}})=A_1^2.
\end{equation}
We have $I_A^{s+1}=0$. Moreover, for each $0 < t \le s$, the products
\[y^J\ \mathrm{mod}~I_A^{t+1}\]
generate the $\overline{A}$-module $I_A^t/I_A^{t+1}$, where $J$ runs over all subsets of $\underline{s}$ of cardinality $t$. 
For (ii) and (iii) of (b) above, it remains to prove that the products of length $t$
are $\overline{A}$-linearly independent. By \eqref{eq:Ann}
this holds when $t=s$.
Assume that an $A$-linear combination 
\[ \sum_{J\subseteq\underline{s}, |J|=t} a_J y^J,\quad  
a_J\in A_0, \] 
belongs to $I_A^{t+1}$. Multiplying by the product $y^L$,
$L=\underline{s}\setminus J$, one obtains $a_J y^{\underline{s}}=0$ for each subset $J$ of cardinality $t$. This, together with the result proven above when $t=s$, show the desired 
result. 

Assume (b). By (ii) we have elements $z_1,\dots, z_s$ in $A_1$ 
which give rise to an $\overline{A}$-free basis of $A_1/A_1^3=I_A/I_A^2$.
The elements form a minimal system of generators of the $A_0$-module
$A_1$, and so $I_A^{s+1}=0$, since they 
give rise to a $K$-basis of $A_1/\mathfrak{m}_0A_1=(\mathfrak{m}/\mathfrak{m}^2)_1$.
By (iii) we have
$\overline{A}\simeq \overline{A} z^{\underline{s}}=I_A^s$, whence $\mathrm{Ann}_{A_0}(z^{\underline{s}})=A_1^2$. 
This together with (i) show (a), in view of
Proposition \ref{prop:Ksdim_regular} .   
\end{proof}
\begin{cor}\label{cor:localization_of_regular_local}
Given a regular local super-ring $A$, the localization $A_{\mathfrak{p}}$ at any prime
$\mathfrak{p}$ of $A$ is regular. 
\end{cor}
\begin{proof}
Write $p$ for $\mathfrak{p}_0$. Assume (b) above. 
Note that the constructions relevant to defining $\lambda_A$ commute with the localization. 
Therefore, (i)~$\overline{A_p}= (\overline{A})_p$ is regular, 
(ii)~$I_{A_p}/I_{A_p}^2=(I_A/I_A^2)_p$ is $\overline{A_p}$-free, and 
(iii)~$\lambda_{A_p}= (\lambda_A)_p$
is an isomorphism. Thus, $A_p$ satisfies (b).  
\end{proof}
We define
\[ \mathsf{gr}_{\mathfrak{m}}(A) 
:= \bigoplus_{n\ge 0} \mathfrak{m}^n/\mathfrak{m}^{n+1} =K \oplus \mathfrak{m}/\mathfrak{m}^2 \oplus \dots.\]
This is a graded superalgebra over $K$. 

Let $S_K((\mathfrak{m}/\mathfrak{m}^2)_0)$ denote the symmetric $K$-algebra on 
the even component of $\mathfrak{m}/\mathfrak{m}^2$; this is a graded algebra.  
We let 
\[ \kappa_A : S_K((\mathfrak{m}/\mathfrak{m}^2)_0) \otimes_K \wedge_K((\mathfrak{m}/\mathfrak{m}^2)_1) 
\to \mathsf{gr}_{\mathfrak{m}}(A) \]
denote the graded $K$-superalgebra map induced by the embedding  
$\mathfrak{m}/\mathfrak{m}^2 \to \mathsf{gr}_{\mathfrak{m}}(A)$. 
One sees that this is a surjection. 

The following theorem generalizes \cite[Theorem 35]{mats} (see also \cite[Theorem 11.22]{AM}).
\begin{theorem}\label{thm:regular_local}
A local super-ring $A$ is regular if and only if  
$\kappa_A$ is an isomorphism. 
\end{theorem}
\begin{proof}
The part "if".\
Assume that $\kappa_A$ is an isomorphism.  
Let $x_1,\dots, x_r$ and $y_1,\dots, y_s$ be 
even and odd elements, respectively, of $\mathfrak{m}$ which all together give rise to
a $K$-basis of $\mathfrak{m}/\mathfrak{m}^2$.
We have $A_1^{s+1}=0$, as was seen before. 

We wish to show that the graded (polynomial) subalgebra $\mathcal{P}$ of $\mathsf{gr}_{\mathfrak{m}}(A)$ which
is freely generated by $x_i~\mathrm{mod}(\mathfrak{m}_0^2+A_1^2)$, $1 \le i \le r$, maps isomorphically
onto $\mathsf{gr}_{\overline{\mathfrak{m}}}(\overline{A})$; this implies that $\overline{A}$ is regular
by \cite[Theorem 35]{mats}. 
Fix $t > 0$. It suffices to show that if an $A_0$-linear combination
\begin{equation}\label{eq:A0-linear_combination}
\sum_{0\le i_1 \le \dots \le i_t \le r}
a_{i_1\dots i_t} x_{i_1}\dots x_{i_t} 
\end{equation}
is in  $\mathfrak{m}_0^{t+1} + A_1^2$, then all the coefficients $a_{i_1\dots i_t}$ are in $\mathfrak{m}_0$. 
This follows since $A_1^2 y^{\underline{s}}=0$, and so the assumption implies
\[ \sum_{0\le i_1 \le \dots \le i_t \le r}
a_{i_1\dots i_t} x_{i_1}\dots x_{i_t}\, y^{\underline{s}} \in \mathfrak{m}^{t+s+1} . \]

The result just proven implies that in $\mathsf{gr}_{\mathfrak{m}}(A)$, no non-zero element in $\mathcal{P}$
annihilates $y^{\underline{s}}~\mathrm{mod}~\mathfrak{m}^{s+1}$. Therefore, 
the set $\mathrm{Ann}_{\overline{A}}(y^{\underline{s}})$ of those elements in $\overline{A}$ which
annihilate $y^{\underline{s}}$ on $A$ is included in $\bigcap_{i\ge 0}\overline{\mathfrak{m}}^i=0$, and so
$\mathrm{Ann}_{A_0}(y^{\underline{s}})=A_1^2$. It follows by Proposition \ref{prop:Ksdim_regular} that
$A$ is regular. 

The part "only if".\ 
First, we wish to see from the results in Section \ref{subsec:completion} that
$A$, and hence $A_0$ as well as $\overline{A}$, with all replaced by their completions,
may be assumed to be complete. 
Indeed, we see $\mathfrak{m}/\mathfrak{m}^2= \widehat{\mathfrak{m}}/\widehat{\mathfrak{m}}^2$,
$\mathsf{gr}_{\mathfrak{m}}(A)=\mathsf{gr}_{\widehat{\mathfrak{m}}}(\widehat{A})$, and that
$\kappa_A$ is identified with $\kappa_{\widehat{A}}$. 
In view of Lemma \ref{propertiesoflocal}
it remains to prove that $\widehat{A}$ is regular, assuming that $A$ is regular. 
By \cite[(24D)]{mats} $\widehat{A}_0$ is regular. 
The desired result follows by Proposition \ref{prop:Ksdim_regular}, 
since one sees, using Proposition \ref{completionofNoetherian}, 
that if $y^{\underline{s}} : A_0 \to A$, the multiplication by $y^{\underline{s}}$,
has $A_1^2$ as the kernel, then 
$y^{\underline{s}} : \widehat{A}_0 \to \widehat{A}$ 
has $\widehat{A_1^2}=(\widehat{A}_1)^2$ as the kernel.  

We may now assume that $A$, $A_0$ and $\overline{A}$ are complete, and satisfy (i)--(iii) of 
(b) in Proposition \ref{prop:regular_local}. Choose the even and odd elements $x_1, \ldots, x_r$ and $y_1, \ldots, y_s$ as above. They are the topological generators of $A$, regarded as a topological ring, i.e. each element $a\in A$ is equal to a (not necessary unique) series
\[\sum_{\alpha\in\mathbb{Z}_{\geq 0}^s, I\subseteq\underline{s}}a_{\alpha, I}x^{\alpha}y^I,\]
where $x^{\alpha}=x_1^{\alpha_1}\ldots x_r^{\alpha_r}$, provided $\alpha=(\alpha_1, \ldots, \alpha_r)$, and $a_{\alpha, I}\in A_0\setminus\mathfrak{m}_0$, whenever $a_{\alpha, I}\neq 0$. 
Such a series is said to be \emph{formally nonzero}, if at least one coefficient $a_{\alpha, I}\neq 0$. 

To complete the proof one has to show that any formally nonzero series represents a nonzero element of $A$. 

Observe that $\overline{A}$ is topologically generated by the elements $x_i+A_1^2, 1\leq i\leq r$, that is each element $\overline{a}$ of $\overline{A}$ is equal to a series
\[\sum_{\alpha\in\mathbb{Z}_{\geq 0}^s}a_{\alpha, I}x^{\alpha} \pmod{A_1^2},\]
and \cite[Theorem 35]{mats} says that $\overline{a}\neq 0$ if and only if $\overline{a}$ is represented by a formally nonzero series in $x$-s. 

Assume that there is a formally nonzero series
\[\sum_{\alpha\in\mathbb{Z}_{\geq 0}^s, I\subseteq\underline{s}}a_{\alpha, I}x^{\alpha}y^I,\]
that represents a zero in $A$. This series can be represented as a sum
$\sum_{I\subseteq\underline{s}}f_I y^I$, 
where among all coefficients $f_I\in A_0$ there is at least one, say $f_J$, which is formally nonzero modulo $A^2_1$. We call $J$ a \emph{good exponent}.

Choose a good exponent $J$ with minimal $|J|$. Multiplying by $y^{\underline{s}\setminus J}$, one obtains an equation
\[f_J y^{\underline{s}}=\overline{f_J}y^{\underline{s}}=0,\]
which obviously contradicts the fact that $I_A^s=A_1^s$ is a free $\overline{A}$-module, freely generated by $y^{\underline{s}}$. 
\end{proof}

Given a filed $K$ and non-negative integers $r$ and $s$, we have the $K$-superalgebra 
\[ K[[X_1,\dots, X_r]]\otimes_K \wedge_K(Y_1,\dots, Y_s), \] 
as above; this is called 
the \emph{formal power series $K$-superalgebra} in even and odd variables, $X_1,\dots, X_r$ and
$Y_1,\dots, Y_s$.  
In fact, this is a complete regular local super-ring by Theorem \ref{thm:regular_local},
since the kappa map is an isomorphism; it is indeed the
identity map on $K[X_1,\dots, X_r] \otimes_K \wedge_K(Y_1,\dots, Y_s)$. 
\begin{rem}
\emph{Theorem \ref{thm:regular_local} coincides with Theorem 3.3 from \cite{sm}, but our proof is quite different and seems more elementary.}
\end{rem}
\begin{cor}\label{cor:completion}(see \cite{sm}, or \cite{maszub2}, A.3)
Let $(A,\mathfrak{m})$ be a local super-ring. 
Assume that $A_0$ includes a field. 
Then the following are equivalent:
\begin{itemize}
\item[(a)] 
$A$ is regular;
\item[(b)] 
$\widehat{A}$ is regular;
\item[(c)]
$\widehat{A}$ is isomorphic to 
$K[[X_1,\dots, X_r]]\otimes_K \wedge_K(Y_1,\dots, Y_s)$, where $K = A/\mathfrak{m}$ and
$r \mid s= \mathrm{sdim}_K(\mathfrak{m}/\mathfrak{m}^2)$. 
\end{itemize}
\end{cor}
\begin{proof}
(c) $\Rightarrow$ (b).\ 
This was just seen above.

(b) $\Rightarrow$ (a).\ 
This follows by Theorem \ref{thm:regular_local}, since $\kappa_{A}$
is identified with $\kappa_{\widehat{A}}$.

(a) $\Rightarrow$ (c).\
As above, one can assume that $A$ is complete. Since $A_0$ includes a field, Cohen's theorem implies that $A$ has a coefficient field (or field of representatives), isomorphic to $K$. Moreover,
$\overline{A}\simeq
K[[X_1, \ldots, X_r]]$ is formally smooth (see \cite[Example 3., p.200; Corollary 2, p.206]{mats}). 
Thus the epimorphism $A_0\to\overline{A}$ splits and (c) follows by Proposition \ref{prop:regular_local}(b).
\end{proof}

\subsection{Regular super-rings}\label{sec:regular_global}
Let $A$ be a super-ring which we continue to assume to be Noetherian. The super-ring $A$ is said to be \emph{regular} if for every prime $\mathfrak{p}$ of $A$, the local super-ring
$A_{\mathfrak{p}}$ is regular. By Corollary \ref{cor:localization_of_regular_local} this is 
equivalent to saying that for every maximal
$\mathfrak{m}$ of $A$, the local super-ring
$A_{\mathfrak{m}}$ is regular. 

\begin{pr}\label{prop:regular_global}
For a super-ring $A$, the following are equivalent:
\begin{itemize}
\item[(a)] $A$ is regular;
\item[(b)]
\begin{itemize} \item[(i)] The ring $\overline{A}$ is regular, 
\item[(ii)] the $\overline{A}$-module $I_A/I_A^2$ is projective, and 
\item[(iii)] $\lambda_A$ is an isomorphism.
\end{itemize}
\end{itemize}
\end{pr}
\begin{proof}
This follows from Proposition \ref{prop:regular_local}, since the conditions above admit
``local criterion" in the sense that each of them is satisfied if and only if 
the condition naturally obtained by localization at every prime/maximal is satisfied.  
\end{proof}
\begin{rem}\label{coincidenceofdefinitions}
\emph{Proposition \ref{prop:regular_global} shows that our definition of regularity is equivalent to the definition of regularity introduced in \cite[3.3, p.79]{sm}. For example, if $A$ is an regular super-ring, no matter local or not, then any minimal system of generators of the $A_0$-module $I_A/I_A^2$ is an odd regular sequence of maximal length.}  
\end{rem}
\begin{rem}\label{acounterexample}
\emph{The converse of Lemma \ref{whensuper-dimensionisdecreasing}(b) is not true, even if $R$ is regular. In fact, let $R=K[Y]$ be a polynomial superalgebra in $s$ odd free generators as in subsection 4.4. Set $f=Y^{L_1}+Y^{L_2}$, where $L_1\not\subseteq L_2, L_2\not\subseteq L_1, L_1\cap L_2\neq\emptyset,$ and $|L_1|\equiv |L_2|\pmod 2$. Proposition \ref{t=2} implies $\mathrm{Ksdim}_1(R/Rf)=s-1$, but 
$Y^{L_1\cap L_2}\in\mathrm{Ann}_R(f)\setminus Rf$.}
\end{rem}

\subsection{Krull super-dimension of completions of super-rings}

We still assume that all super-rings are Noetherian. 

\begin{lm}\label{a partial case}
Let $B$ be a Noetherian super-ring with $\mathrm{Kdim}(B)<\infty$. If $\mathfrak{p}$ is a prime super-ideal of $B$ such that $\mathrm{Ksdim}_0(B)=\mathrm{Ksdim}_0(B_{\mathfrak{p}})$, then $\mathrm{Ksdim}_1(B_{\mathfrak{p}})\leq \mathrm{Ksdim}_1(B)$.
\end{lm}
\begin{proof}
The condition $\mathrm{Ksdim}_0(B)=\mathrm{Ksdim}_0(B_{\mathfrak{p}})$ is equivalent to $\mathrm{ht}(\mathfrak{p}_0)=\mathrm{Kdim}(B_0)$. In other words, any longest prime chain $\mathfrak{q}_0\subseteq\ldots\subseteq\mathfrak{q}_r=\mathfrak{p}_0$ is a longest prime chain in $B_0$ as well.
Thus the elements $\frac{z_1}{s_1}, \ldots , \frac{z_t}{s_t}$, where $z_1, \ldots , z_t\in B_1, s_1, \ldots, s_t\in S=B_0\setminus\mathfrak{p}_0$, form a system of odd parameters in $B_{\mathfrak{p}}$ if and only if
\[\mathrm{Ann}_{S^{-1}B_0}(\frac{z_1}{s_1} \ldots \frac{z_t}{s_t})=S^{-1}\mathrm{Ann}_{B_0}(z_1 \ldots s_t)\subseteq S^{-1}\mathfrak{q}_0\]
if and only if
\[\mathrm{Ann}_{B_0}(z_1 \ldots s_t)\subseteq \mathfrak{q}_0\]
for an ideal $\mathfrak{q}_0$ as above. Now the statement is obvious.
\end{proof}
\begin{theorem}\label{an equality of super-dimensions}
Let $A$ be a local Noetherian super-ring with maximal super-ideal $\mathfrak{m}$. Let $\widehat{A}$ denote its $\mathfrak{m}$-adic completion. Then
$\mathrm{Ksdim}(A)=\mathrm{Ksdim}(\widehat{A})$.
\end{theorem} 
\begin{proof}
Since $\widehat{A}_0$ coincides with the $\mathfrak{m}_0$-adic completion of $A_0$, the equality $\mathrm{Ksdim}_0(A)=\mathrm{Ksdim}_0(\widehat{A})$ follows by (24.D)(i), \cite{mats}.
Fix a generating set $y_1, \ldots , y_d$ of $A_0$-module $A_1$.

Let $b$ be a homogeneous element from $A$. Arguing as in Theorem \ref{thm:regular_local}, we obtain an exact sequence 
\[0\to\widehat{\mathrm{Ann}_{A_0}(b)}\to \widehat{A}_0\to \widehat{A},\]
where $\widehat{A}_0\to \widehat{A}$ is defined as $a\mapsto ab, a\in \widehat{A}_0$. 
In other words, we have
$\mathrm{Ann}_{\widehat{A}_0}(b)=\widehat{\mathrm{Ann}_{A_0}(b)}$. Furthermore, $\widehat{A}_0/\mathrm{Ann}_{\widehat{A}_0}(b)$ is isomorphic to the $\mathfrak{m}_0/\mathrm{Ann}_{A_0}(b)$-adic completion of  $A_0/\mathrm{Ann}_{A_0}(b)$, that in turn infers
\[\mathrm{Kdim}(\widehat{A}_0/\mathrm{Ann}_{\widehat{A}_0}(b))=\mathrm{Kdim}(A_0/\mathrm{Ann}_{A_0}(b)).\]
By Theorem 55, \cite{mats}, $\widehat{A}_1$ is generated by the same elements $y_1, \ldots, y_d$ as an $\widehat{A}_0$-module, and Proposition \ref{another definition of odd super-dimension} completes the proof.
\end{proof}
The set of super-dimension vectors can be lexicographically ordered as  $a|b < c|d$, provided $a< b$ or $a=b$ and then $c< d$. The following theorem superizes (24.D)($\mathrm{i}'$), \cite{mats}.
\begin{theorem}\label{super-dim of completion}
Let $A$ be a Noetherian super-ring with $\mathrm{Ksdim}_0(A)<\infty$. Let $I$ be a super-ideal of $A$ and let $\widehat{A}$ be its $I$-adic completion. Then 
\[\mathrm{Ksdim}(\widehat{A})=\sup_{\mathfrak{p}\in SSpec(A), I\subseteq\mathfrak{p}}\mathrm{Ksdim}(A_{\mathfrak{p}}).\] 
\end{theorem}
\begin{proof}
Using Lemma \ref{coincidenceoftopologies} and (24.D)($\mathrm{i}'$), \cite{mats}, we have 
\[r=\mathrm{Ksdim}_0(\widehat{A})=\sup_{\mathfrak{q}\in Spec(A_0), I_0\subseteq\mathfrak{q}}\mathrm{Kdim}((A_0)_{\mathfrak{q}})=\sup_{\mathfrak{p}\in SSpec(A), I\subseteq\mathfrak{p}}\mathrm{Ksdim}_0(A_{\mathfrak{p}}).\]
As above, $\widehat{A}_0$-module $\widehat{A}_1$ is generated by some $y_1, \ldots, y_d\in A_1$. By Proposition \ref{another definition of odd super-dimension} there is a system of odd parameters $y_{i_1}, \ldots , y_{i_s}$ of $\widehat{A}$ of largest cardinality $s=\mathrm{Ksdim}_1(\widehat{A})$.  

Let $J$ denote the ideal $\mathrm{Ann}_{A_0}(y_{i_1}\ldots y_{i_s})$. Arguing as in Theorem \ref{an equality of super-dimensions}, we have
\[\widehat{A}_0/\mathrm{Ann}_{\widehat{A}_0}(y_{i_1}\ldots y_{i_s})\simeq \widehat{A_0/J},\]
where $\widehat{A_0/J}$ denotes the completion of $A_0/J$ in the $(I+J)/J$-adic topology. Thus 
\[\mathrm{Kdim}(\widehat{A}_0)=\mathrm{Kdim}(\widehat{A}_0/\mathrm{Ann}_{\widehat{A}_0}(y_{i_1}\ldots y_{i_s}))=\mathrm{Kdim}(\widehat{A_0/J}).\]

Using (24.D)($\mathrm{i}'$), \cite{mats}, one sees
\[r=\sup_{\mathfrak{q}\in Spec(A_0), I_0\subseteq\mathfrak{q}}\mathrm{Kdim}((A_0)_{\mathfrak{q}})=
\sup_{\mathfrak{q}\in Spec(A_0), I_0+J\subseteq\mathfrak{q}}\mathrm{Kdim}((A_0/J)_{\mathfrak{q}/J}).\]
There are two prime ideals $\mathfrak{q}$ and $\mathfrak{q}'$ such that $I_0\subseteq\mathfrak{q}, I_0+J\subseteq\mathfrak{q}'$, and
\[r=\mathrm{Kdim}((A_0)_{\mathfrak{q}})=\mathrm{Kdim}((A_0/J)_{\mathfrak{q}'/J}).\]
On the other hand, we have
\[\mathrm{Kdim}((A_0)_{\mathfrak{q}})\geq \mathrm{Kdim}((A_0)_{\mathfrak{q}'})\geq
\mathrm{Kdim}((A_0)_{\mathfrak{q}'}/J_{\mathfrak{q}'})=\mathrm{Kdim}((A_0/J)_{\mathfrak{q}'/J}),\]
hence
\[r=\mathrm{Kdim}((A_0)_{\mathfrak{q}})= \mathrm{Kdim}((A_0)_{\mathfrak{q}'})=\mathrm{Kdim}((A_0)_{\mathfrak{q}'}/J_{\mathfrak{q}'})\]
and the elements $y_{i_1}, \ldots, y_{i_s}$ form a system of odd parameters in $(A_0)_{\mathfrak{q}'}$. The latter infers the inequality
\[\mathrm{Ksdim}_1(\widehat{A})\leq\mathrm{Ksdim}_1((A)_{\mathfrak{q}'+A_1}),\]
and therefore, the inequality
\[\mathrm{Ksdim}(\widehat{A})\leq\sup_{\mathfrak{p}\in SSpec(A), I\subseteq\mathfrak{p}}\mathrm{Ksdim}(A_{\mathfrak{p}}).\]
Let $\mathfrak{p}$ be a prime super-ideal such that $I\subseteq\mathfrak{p}$ and $\mathrm{Ksdim}_0(\widehat{A})=\mathrm{Ksdim}_0(A_{\mathfrak{p}})$. Arguing as in (24D), \cite{mats}, one can show that $A_{\mathfrak{p}}$ and $\widehat{A}_{\widehat{\mathfrak{p}}}$ are {\it analytically isomorphic}, that is their completions (as local super-rings) are isomorphic. Combining Theorem \ref{an equality of super-dimensions} and Lemma \ref{a partial case}, one obtains  
\[\mathrm{Ksdim}(A_{\mathfrak{p}})=\mathrm{Ksdim}(\widehat{A}_{\widehat{\mathfrak{p}}})\leq\mathrm{Ksdim}(\widehat{A}),\]
hence
\[\sup_{\mathfrak{p}\in SSpec(A), I\subseteq\mathfrak{p}}\mathrm{Ksdim}(A_{\mathfrak{p}})\leq\mathrm{Ksdim}(\widehat{A}).\] 
Theorem is proven.
\end{proof}

\subsection{K\"{a}hler superdifferentials and regularity}

\begin{lm}\label{cotangent}
Let $B$ be a local $K$-superalgebra with the maximal super-ideal $\mathfrak{m}$.
Assume that its residue field $B/\mathfrak{m}=K(B)$ is a separably generated extension of $K$. Then  
\[\Omega_{B/K}\otimes_B K(B)\simeq \mathfrak{m}/\mathfrak{m}^2\oplus \Omega_{K(B)/K}. \]
In particular, if $B$-supermodule $\Omega_{B/K}$ is finitely generated, then a minimal system of generators of $\Omega_{B/K}$ consists of 
\[\mathrm{dim}_{K(B)}((\mathfrak{m}/\mathfrak{m}^2)_0)+\mathrm{tr.deg}_K(K(B))\] 
even generators and 
\[\mathrm{dim}_{K(B)}((\mathfrak{m}/\mathfrak{m}^2)_1)\]
odd generators respectively.
\end{lm}
\begin{proof}
Applying \cite[p.205]{mats} to a complete local ring $(B/\mathfrak{m}^2)_0$, one can choose its field of representatives, say $L\simeq K(B)$, so that the exact sequence 
\[0\to\mathfrak{m}/\mathfrak{m}^2\to B/\mathfrak{m}^2\to K(B)\to 0\]
is split (on the right). Proposition \ref{secondexactsequence} and \cite[Theorem 59]{mats} conclude the proof. 
\end{proof}
\begin{lm}\label{acriteria}
Let $A$ be a Noetherian local superdomain with residue field $K$ and superfield of fractions $F$.
Let $M$ be a finite $A$-supermodule such that the $F$-supermodule $M_{I_A}=M\otimes_A F$ is free of rank $p|q$. If $\mathrm{sdim}_K M\otimes_A K=p|q$, then $M$ is a free $A$-supermodule of rank $p|q$.  
\end{lm}
\begin{proof}
Since $F$ is a flat $A$-supermodule (cf. \cite{maszub1}, Lemma 1.2(i)), one can easily superize the proof of Lemma II.8.9, \cite{hart}. 
\end{proof}
\begin{theorem}\label{regularity}
Assume that $B$ is a $K$-superalgebra as above. Assume also that $K$ is perfect and $B$ is a localization of a finitely generated $K$-superalgebra. Then $B$ is regular if and only if $\Omega_{B/K}$ is a free $B$-supermodule of rank equal to $\mathrm{Ksdim}(B)+\mathrm{tr.deg}_K(K(B))|0$.
\end{theorem}
\begin{proof}
First of all, Remark \ref{finiteness} and Lemma \ref{baseextension} imply that $\Omega_{B/K}$ is a finitely generated $B$-supermodule.

Set $\mathrm{sdim}_{K(B)}(\mathfrak{m}/\mathfrak{m}^2)=m|n$ and $\mathrm{Ksdim}(B)+0|\mathrm{tr.deg}_K(K(B))=p|q$. Note that $K(B)$ is a finitely generated extension of $K$, hence it is also separably generated over $K$ (see \cite{mats}, p.194). If $\Omega_{B/K}$ is free of rank $p|q$, then Lemma \ref{cotangent} implies $\mathrm{Ksdim}(B)=m|n$, whence $B$ is regular.

Conversely, assume that $B$ is regular, i.e. the Krull superdimension of $B$ is equal to $m|n$. Again, by Lemma \ref{cotangent} the $K(B)$-superspace $\Omega_{B/K}\otimes_B K(B)$ has superdimension $p|q$. 

Proposition \ref{prop:regular_local} implies that $B$ is a strong superdomain. 
The superfield $F=SQ(B)$ is a local Noetherian super-ring with the maximal super-ideal $I_F$. Moreover, by Corollary \ref{cor:localization_of_regular_local} the super-ring $F$ is regular.
The residue field of $F$, $K(F)$, is isomorphic to the quotient field of $\overline{B}$.

Further, the maximal super-ideal $I_F$ is nilpotent, hence $F$ is complete. Since
$\mathrm{Ksdim}_1(F)=\mathrm{Ksdim}_1(B)=n$ (see the proof of Proposition \ref{dimension} below), Corollary \ref{cor:completion} shows that $F\simeq K(F)[Y_1, \ldots, Y_n]$. Observe that any $K$-superderivation of $K(F)$ into a $F$-supermodule $T$ can be extended to a $K$-superderivation $F\to T$. Therefore, Proposition \ref{firstexactsequence} infers
\[\Omega_{F/K}\simeq (\Omega_{K(F)/K}\otimes_{K(F)} F)\oplus\Omega_{F/K(F)}.\]
By Remark \ref{anexample}, $\Omega_{F/K(F)}$ is a free $F$-supermodule of rank $0|n$. On the other hand, Lemma \ref{baseextension} implies 
\[\Omega_{K(F)/K}\simeq\Omega_{\overline{B}/K}\otimes_{\overline{B}} K(F).\]
Since $\Omega_{\overline{B}/K}$ is a free $\overline{B}$-module of rank $m+\mathrm{tr.deg}_K(K(B))$ (see \cite[Exercise II.8.1(b)]{hart}), $\Omega_{F/K}$ is a free $F$-supermodule of rank $p|q$.
Again, by Lemma \ref{baseextension} there is
\[\Omega_{B/K}\otimes_B F\simeq\Omega_{F/K},\]
and Lemma \ref{acriteria} concludes the proof.
\end{proof}

\section{Dimension theory of superschemes}

\subsection{Super-dimension of a superscheme}

From now on all superschemes are assumed to be of finite type over a field $K$, unless otherwise stated.
\begin{pr}\label{dimension}
Let $X$ be an irreducible superscheme. If $U$ and $V$ are non-empty open affine super-subschemes of $X$, then $\mathrm{Ksdim}(\mathcal{O}(U))=\mathrm{Ksdim}(\mathcal{O}(V))$.
\end{pr}
\begin{proof}
Since $U^e\cap V^e$ is not empty, all we need is to consider the case $U\subseteq V$. Set $U\simeq SSpec \ A$ and $V\simeq SSpec \ B$. Then the natural open immersion $U\to V$ coincides with 
$SSpec \ \phi$ for some super-ring morphism $\phi : B\to A$ . By Lemma 3.5, \cite{maszub1}, there are
$b_1, \ldots, b_t\in B_0$ such that $\sum_{1\leq i\leq t}A_0\phi(b_i)=A_0$ and the induced morphisms
$B_{b_i}\to A_{\phi(b_i)}$ are isomorphisms. Thus the general case can be reduced to $V=SSpec \ B$ and $U=SSpec \ B_b , b\in B_0\setminus \mathrm{nil}(B_0)$. Observe that $\mathrm{nil}((B_0)_b)=\mathrm{nil}(B_0)_b$.

Since $B_0/\mathsf{nil}(B_0)$ is a domain, $\mathrm{Ksdim}_0(B) =\mathrm{Kdim}(B_0/\mathrm{nil}(B_0))$ coincides with the transcendence degree of its field of fractions (cf. \cite{AM}, XI), whence $\mathrm{Ksdim}_0(B) =\mathrm{Ksdim}_0(B_b)$. Lemma \ref{whensuperdimensiondoesnotchange} concludes the proof.
\end{proof}
Proposition \ref{dimension} allows to define a {\it super-dimension} of any irreducible superscheme $X$  as $\mathrm{sdim}(X)=\mathrm{Ksdim}(\mathcal{O}(U))$, where $U$ is any (non-empty) open affine super-subscheme of $X$. 
\begin{rem}\label{superdimatclosedpoint}
\emph{Let $X$ be an irreducible superscheme. The same arguments as in Proposition \ref{dimension} show that $\mathrm{Ksdim}_1(\mathcal{O}_x)=\mathrm{sdim}_1(X)$ for any point $x\in X^e$. In particular, if $x$ is a closed point of $X$, then $\mathrm{Ksdim}(\mathcal{O}_x) =\mathrm{sdim}(X)$. Besides, we have $\mathrm{sdim}(U)=\mathrm{sdim}(X)$ for (nonempty) open super-subscheme $U$ of $X$.}
\end{rem}
\subsection{Nonsingular superschemes}

Let $X$ be an irreducible superscheme. Then $X$ is said to be \emph{nonsingular}, if for any $x\in X^e$ the super-ring $\mathcal{O}_x$ is regular.
\begin{theorem}\label{acriteriaofnonsingularity}
Let $X$ be as above and assume that $K$ is perfect. Then $X$ is nonsingular if and only if the sheaf $\Omega_{X/K}$ is locally free of rank $\mathrm{sdim}(X)$.
\end{theorem}
\begin{proof}
Without loss of generality one can assume that $X$ is affine, say $X=SSpec \ R$. Then $\mathrm{Ksdim}(R)=\mathrm{sdim}(X)$ and all one need to prove is that $R_{\mathfrak{p}}$ is regular if and only if $(\Omega_{R/K})_{\mathfrak{p}}\simeq \Omega_{R_{\mathfrak{p}}/K}$ is a free $R_{\mathfrak{p}}$-supermodule of rank $\mathrm{Ksdim}(R)$, for each point $\mathfrak{p}\in (SSpec \ R)^e$. 

On the other hand, by Theorem \ref{regularity} the local super-ring $R_{\mathfrak{p}}$ is regular if and only if  $\Omega_{R_{\mathfrak{p}}/K}$ is a free $R_{\mathfrak{p}}$-supermodule of rank $\mathrm{Ksdim}(R_{\mathfrak{p}})+\mathrm{tr.deg}_K(K(R_{\mathfrak{p}}))|0$. Since $\mathrm{Ksdim}_1(R)=\mathrm{Ksdim}_1(R_{\mathfrak{p}})$, the equality 
\[\mathrm{Ksdim}(R)=\mathrm{Ksdim}(R_{\mathfrak{p}})+\mathrm{tr.deg}_K(K(R_{\mathfrak{p}}))|0\]
holds if and only if
\[\mathrm{Ksdim}_0(R)=\mathrm{Ksdim}_0(R_{\mathfrak{p}})+\mathrm{tr.deg}_K(K(R_{\mathfrak{p}}))\]
does. Note that
\[\mathrm{Ksdim}_0(R)=\mathrm{Kdim}(\overline{R}), \quad \mathrm{Ksdim}_0(R_{\mathfrak{p}})=\mathrm{Kdim}(\overline{R}_{\overline{\mathfrak{p}}}).\]
Moreover, since $K(R_{\mathfrak{p}})=K(\overline{R}_{\overline{\mathfrak{p}}})$, we have
\[\mathrm{tr.deg}_K(K(R_{\mathfrak{p}}))=\mathrm{tr.deg}_K(K(\overline{R}_{\overline{\mathfrak{p}}})).\]
Thus any of the above equalities holds if and only if $\overline{R}_{\overline{\mathfrak{p}}}$ is regular (see \cite{hart}, Exercise II.8.1.c). Theorem obviously follows.
\end{proof}

\subsection{Generically nonsingular superschemes} \

Oppositely to the purely even case, there are (even strong) integral superschemes, which are singular at any point (compare with \cite{hart}, Exercise II.8.1(d)). 
\begin{example}\label{everywhere singular}
\emph{In fact, let $A$ be a finitely generated $K$-superalgebra such that $A_0$ is a domain and $A_1^2=0$.
Assume also that $A_1$ is a free $A_0$-module of rank $t > 1$.
Set $X=SSpec \ A$.}

\emph{Since $A$ is a strong superdomain, $X$ is strong integral. Further, $(A_{\mathfrak{p}})_1$ is a free $(A_{\mathfrak{p}})_0$-module of the same rank $t$, for each $\mathfrak{p}\in X^e$. Proposition
\ref{prop:Ksdim_regular}(b) immediately shows that $A_{\mathfrak{p}}$ is not regular.}
\end{example}
A superscheme $X$ is called \emph{generically nonsingular}, provided $X$ contains an nonempty open nonsingular super-subscheme.
\begin{lm}\label{opennonsingular}
Let $X$ be an irreducible superscheme over a perfect field $K$. Then $X$ is generically nonsingular if and only if there is $x\in X^e$ such that $\mathcal{O}_x$ is regular.
\end{lm} 
\begin{proof}
As it has been proven in Theorem \ref{acriteriaofnonsingularity}, a local super-ring $\mathcal{O}_x$ is regular if and only if $(\Omega_{X/K})_x$ is a free $\mathcal{O}_x$-supermodule of rank $\mathrm{sdim}(X)$. Thus our statement follows by Lemma \ref{locallyfreesheaves} and Remark \ref{superdimatclosedpoint}.
\end{proof}
\begin{pr}\label{regularonopensubset}
Let $X$ be an integral superscheme over a perfect field $K$. Let $\xi$ be the generic point of $X$. The following conditions are equivalent :
\begin{itemize}
\item[(a)] $X$ is generically nonsingular;
\item[(b)] $\mathcal{O}_{\xi}$ is regular;
\item[(c)] there is a nonempty open super-subscheme $U$ of $X$ such that $U_{ev}$ is regularly immersed into $U$ (cf. \cite[4.5]{sm}).
\end{itemize}
\end{pr}
\begin{proof}
Use Lemma \ref{opennonsingular} to prove (a)$\Leftrightarrow$ (b).

Without loss of generality one can assume that $X$ is affine, say $X\simeq SSpec \ A$. 
Recall that $\mathcal{O}_{\xi}$ is a superfield with the maximal superieal $\mathfrak{m}_{\xi}=I_{\mathcal{O}_{\xi}}$. Thus $\mathcal{O}_{\xi}$ is regular if and only if $\lambda_{\mathcal{O}_{\xi}} : \wedge_{\overline{\mathcal{O}_{\xi}}}(\mathcal{I}_{\xi}/\mathcal{I}_{\xi}^2)\to\mathsf{gr}_{\mathcal{I}_{\xi}}(\mathcal{O}_{\xi})$ is an isomorphism if and only $I_{\mathcal{O}_{\xi}}$ is an regular superideal (see \cite[3.2]{sm}). 

Since both $\wedge_{\overline{A}}(I_A/I_A^2)$ and $\mathsf{gr}_{I_A}(A)$ are finitely generated $\overline{A}$-(super)modules, arguing as in Lemma \ref{locallyfreesheaves} one can show that 
$\lambda_{\mathcal{O}_{\xi}}$ is an isomorphism if and only if
there is an open subset $U\subseteq X^e$, such that for any $\mathfrak{p}\in U$ the following hold :
\begin{itemize}
\item[(1)]  $\lambda_{\mathcal{O}_{\mathfrak{p}}}$  is an isomorphism;
\item[(2)] $\mathcal{I}_{\mathfrak{p}}/\mathcal{I}_{\mathfrak{p}}^2$ is a free $\overline{\mathcal{O}_{\mathfrak{p}}}$-supermodule,
\end{itemize}
whence (b)$\Leftrightarrow$(c). 
\end{proof}
Recall that if $X$ is an integral scheme and $Y$ is a closed integral subscheme of $X$, such that $\mathrm{dim}(X)=\mathrm{dim}(Y)$, then $X=Y$. Surprisingly, a naive analog of this statement is no longer true in the category of superschemes. 
\begin{example}\label{no coincidence}
\emph{Let $X=SSpec \ A$ be the everywhere singular superscheme from Example \ref{everywhere singular}. Let $b\in A_1$ is a free generator of $A_0$-module $A_1$. Set $Y=SSpec \ A/Ab$. Then $Y$ is isomorphic to a proper (strong integral as well) closed super-subscheme of $X$, but $\mathrm{sdim}(Y)=\mathrm{sdim}(X)$.}
\end{example}
Nevertheless, a weaker super-analog of the above statement takes place.
\begin{theorem}\label{coincidence}
Let $X$ be a generically nonsingular integral superscheme over a perfect field $K$. If $Y$ is a closed super-subscheme of $X$, which is generically nonsingular and integral as well, then $\mathrm{sdim}(X)=\mathrm{sdim}(Y)$ implies $X=Y$.
\end{theorem} 
\begin{proof}
The purely even version of our theorem infers that $X_{res}=Y_{res}$, that is $\mathcal{J}_Y\subseteq \mathcal{I}_X$. Thus $\xi\in Y^e$, where $\xi$ is a generic point of $X$. Since both $X$ and $Y$ are irreducible, Proposition \ref{regularonopensubset} allows to assume that both $X$ and $Y$ are nonsingular. Moreover, one can also assume that both $X$ and $Y$ are affine, say $X=SSpec \ A$ and $Y\simeq SSpec \ A/I$, where $I\subseteq I_A$ and both $A$ and $A/I$ are regular. 

Let $B$ denote the superalgebra $A/I$. Note that any prime super-ideal of $B$ has a form $\overline{\mathfrak{p}}=\mathfrak{p}/I$, where $\mathfrak{p}$ is a prime super-ideal of $A$. Thus the residue fields of both $A_{\mathfrak{p}}$ and $B_{\overline{\mathfrak{p}}}$ are isomorphic, say, to a field $L$.

For any prime super-ideal $\mathfrak{p}$ of $A$ we have
\[\mathrm{sdim}_1(X)=\mathrm{Ksdim}_1(A_{\mathfrak{p}})=
\mathrm{sdim}_1(Y)=\mathrm{Ksdim}_1(B_{\overline{\mathfrak{p}}}),\]
hence both $\overline{A}\simeq\overline{B}$-modules $I_A/I_A^2$ and $I_B/I_B^2$ are projective modules of the same rank, hence isomorphic. By Proposition \ref{prop:regular_global}(b), the epimorphism $A\to B$ induces an isomorphism  $\mathsf{gr}_{I_A}(A)\simeq\mathsf{gr}_{I_B}(B)$, hence $A\to B$ is an isomorphism and $X=Y$.
\end{proof}

\subsection{Closed super-subschemes of generically nonsingular superschemes}

The following theorem superizes Theorem II.8.17, \cite{hart}.
\begin{theorem}\label{differentialsheavesandregularity}
Let $X$ be a nonsingular irreducible superscheme of finite type over a perfect field $K$. Let $Y$ be an irreducible closed super-subscheme of $X$ defined by a sheaf of superideals $\mathcal{J}$. Then $Y$ is nonsingular if and only if the following conditions hold :
\begin{itemize}
\item[(1)] $\Omega_{Y/K}$ is locally free;
\item[(2)] The sequence 
\[0\to\mathcal{J}/\mathcal{J}^2\to\Omega_{X/K}\otimes_{\mathcal{O}_X}\mathcal{O}_Y\to\Omega_{Y/K}\to 0\]
is exact. Moreover, the sheaf $\mathcal{J}/\mathcal{J}^2$ is locally free of rank $\mathrm{sdim} X-\mathrm{sdim} Y$.
\end{itemize} 
\end{theorem}
\begin{proof}
Assume that both (1) and (2) hold. Let $\mathrm{sdim}(X)=m|n$. For any point $y\in Y^e$ let $A$ and $B$ denote the local superalgebras $\mathcal{O}_{X, y}$ and $\mathcal{O}_{Y, y}$ respectively. The stalk $\mathcal{J}_y$ is naturally identified with the kernel of the local morphism $A\to B$, say $J$, so that there is an exact sequence
\[0\to J/J^2\to\Omega_{A/K}\otimes_A B\to\Omega_{B/K}\to 0,\]
where $A$-supermodule $\Omega_{A/K}$ is free of rank $m|n$ and $B$-supermodule is free of rank $p|q$. 
The latter implies that the above sequence splits, hence $J/J^2$ is a free $B$-supermodule of rank $(m-p)|(n-q)$. By \cite[Theorem 3.5]{sm} the super-ideal $J$ is generated by a regular sequence consisting of $m-p$ even and $n-q$ odd elements. Combining \cite[(15.F), Lemma 4]{mats} with Lemma \ref{whensuper-dimensionisdecreasing}, one obtains 
\[\mathrm{Ksdim}(B)=\mathrm{Ksdim}(A)-(m-p)|(n-q).\]
Since $X$ is nonsingular and $K(B)$ is an extension of $K(A)$, there hold
\[(m|n)=\mathrm{Ksdim}(A)+\mathrm{tr.deg}_K(K(A))\]
and 
\[\mathrm{tr.deg}_K(K(A))\leq \mathrm{tr.deg}_K(K(B)).\]
Combining with Lemma \ref{cotangent} and Lemma \ref{lem:odd_generators}, one derives
\[p|q\geq \mathrm{Ksdim}(B)+\mathrm{tr.deg}_K(K(B))\geq\]
\[\mathrm{Ksdim}(A)-(m-p)|(n-q)+\mathrm{tr.deg}_K(K(A))= p|q,\]
and by Theorem \ref{regularity} one derives that $B$ is regular.

Conversely, assume that $Y$ is nonsingular of superdimension $p|q$. Arguing as in Theorem II.8.17, \cite{hart}, one can construct a closed (nonsingular) super-subscheme $Y'$ of $X$, such that $Y\subseteq Y'$ and $Y'$ has the same superdimension $p|q$. Theorem \ref{coincidence} concludes the proof.
\end{proof}
Let $X$ be a generically nonsingular irreducible superscheme of finite type over a perfect field $K$. Let $U$ denote the open subset $\{x\in X^e\mid \mathcal{O}_x \ \mbox{is \ regular}\}$. It is clear that $U$ is the largest open nonsingular super-subscheme of $X$.
\begin{cor}\label{finalcorollary}
Let $X$ and $U$ be as above. Let $Y$ be a closed irreducible super-subscheme of $X$ with $Y^e\cap U^e\neq\emptyset$. Then $Y$ is generically nonsingular if and only there is a point $y\in Y^e$ such that
\begin{itemize}
\item[(1)] $\Omega_{\mathcal{O}_{Y, y}/K}$ is free;
\item[(2)] The sequence 
\[0\to\mathcal{J}_y/\mathcal{J}_y^2\to\Omega_{\mathcal{O}_{X, y}/K}\otimes_{\mathcal{O}_{X, y}}\mathcal{O}_{Y, y}\to\Omega_{\mathcal{O}_{Y, y}/K}\to 0\]
is exact. 
\end{itemize} 
\end{cor}
\begin{proof}
Use Lemma \ref{locallyfreesheaves} and the above theorem.
\end{proof}

\section{Some questions and open problems}
\subsection{One relation superalgebras}
Let $A=K[X|Y]$ and $f\in I_A$. Let $B$ denote the one relation superalgebra $A/Af$, as in subsection 4.4. The discussion therein rises the following natural question.
\begin{question}\label{exponents}
Is the odd Krull dimension of one relation superalgebra $A/Af$ is determined by the basement of $f$?
\end{question}

\subsection{Gr\"{o}bner-Shirshov basis method}

One of the most powerful tools in the theory of polynomial ideals is the Gr\"{o}bner-Shirshov basis method
(cf. \cite{adamsloust, buch}). The following question is also motivated by the discussion in subsection 4.4.
\begin{question}\label{super-Grobner?}
Does any super-version of Gr\"{o}bner-Shirshov basis method exist?
\end{question}
 
\subsection{Dimension theory of supermodules}

Let $R$ be a Noetherian super-ring and $M$ be a finite $R$-supermodule. One can define a super-dimension of $M$ as 
\[\mathrm{sdim}(M)=\mathrm{Ksdim}(R/\mathrm{Ann}(M)).\]
\begin{problem}\label{super-dimtheoryforsupermodules}
Develop a dimension theory of supermodules over Noetherian super-rings.
\end{problem}

\end{document}